\documentclass[10pt]{amsart}
\usepackage{amsmath,amssymb, amsbsy}
\usepackage[dvips]{graphicx}
\usepackage{color}
\usepackage[latin1]{inputenc}
\usepackage[active]{srcltx}
\newcommand{\R}{{\mathbb R}}
\newcommand{\ren}{\R^{N}}

\renewcommand{\l }{\lambda }

\newcommand{\p}{\partial}
\renewcommand{\O}{\Omega }
\newcommand{\g}{\gamma }
\newcommand{\bra}{\langle}
\newcommand{\ket}{\rangle}
\newcommand{\iy}{\infty}

\newcommand{\onn}{\text{  on   }}
\newcommand{\inn}{\text{  in   }}
\newcommand{\dyle}{\displaystyle}
\newcommand{\dint}{\dyle\int}

\newcommand{\e }{\varepsilon}

\renewcommand{\ge }{\geqslant}
\renewcommand{\geq }{\geqslant}
\renewcommand{\le }{\leqslant}
\renewcommand{\leq }{\leqslant}

\newenvironment{pf}{    {\sc Proof}.\enspace}{\hfill\qed\medskip}
\newtheorem{Theorem}{Theorem}[section]
\newtheorem{Corollary}[Theorem]{Corollary}
\newtheorem{Lemma}[Theorem]{Lemma}
\newtheorem{Proposition}[Theorem]{Proposition}
\theoremstyle{definition} \newtheorem{Definition}[Theorem]{Definition}

\newtheorem{remark}[Theorem]{Remark}
\begin{document}

\title[Fractional heat equation involving Hardy Potential]
{Optimal results for the fractional heat equation involving the Hardy potential}
\thanks{Work partially supported by Project MTM2013-40846-P, MINECO, Spain}
\author[B. Abdellaoui, M. Medina, I. Peral, A. Primo ]{Boumediene abdellaoui, Mar\'{i}a Medina, Ireneo Peral, Ana Primo }

\address{\hbox{\parbox{5.7in}{\medskip     {B. Abdellaoui, Laboratoire d'Analyse Nonlin\'eaire et Math\'ematiques
Appliqu\'ees. \hfill \break\indent D\'epartement de
Math\'ematiques, Universit\'e Abou Bakr Belka\"{\i}d, Tlemcen,
\hfill\break\indent Tlemcen 13000, Algeria.}}}}
\address{\hbox{\parbox{5.7in}{\medskip    {M. Medina, I. Peral and A. Primo, Departamento de Matem\'aticas,\\ Universidad Aut\'onoma de Madrid,\\
        28049, Madrid, Spain. \\[3pt]
        \em{E-mail addresses: }\\{\tt boumediene.abdellaoui@uam.es, \tt ireneo.peral@uam.es, maria.medina@uam.es, ana.primo@uam.es
         }.}}}}

\date{\today}


%%%%%%%%%%%%%%%%%%%%%%%%%%%%%%%%%%%%%%%%%%%%%%%
 \begin{abstract}
 In this paper we study the influence of the Hardy potential in the fractional heat equation. In particular, we consider the problem
$$(P_\theta)\quad
\left\{
\begin{array}{rcl}
 u_t+(-\Delta)^{s} u&=&\l\dfrac{\,u}{|x|^{2s}}+\theta u^p+ c f\mbox{ in } \Omega\times (0,T),\\
u(x,t)&>&0\inn \Omega\times (0,T),\\
u(x,t)&=&0\inn (\ren\setminus\Omega)\times[ 0,T),\\
u(x,0)&=&u_0(x) \mbox{ if }x\in\O,
\end{array}
\right.
$$
where $N> 2s$, $0<s<1$,  $(-\Delta)^s$ is the fractional Laplacian
of order $2s$, $p>1$, $c,\l>0$, $\theta=\{0,1\}$, and  $u_0,\, f\ge 0$ are in a
suitable class of functions.

The main results in the article are:
\begin{enumerate}
\item  Optimal results about \emph{existence} and \emph{instantaneous and complete blow up} in the linear problem $(P_0)$, where
the best constant in the fractional Hardy inequality, $\Lambda_{N,s}$, provides the threshold between existence and nonexistence.
To obtain local sharp estimates of the solutions it is required to prove a weak Harnack inequality for a weighted operator that
appears in a natural way.
\item  The existence of a critical power $p_+(s,\lambda)$ in the semilinear problem $(P_1)$ such that:
\begin{enumerate}
\item If $p> p_+(s,\lambda)$, the problem has no weak positive
supersolutions and a phenomenon of \emph{complete and instantaneous
blow up} happens.
\item If $p< p_+(s,\lambda)$, there exists a
positive solution  for a suitable class of nonnegative data.
\end{enumerate}
\end{enumerate}
 \end{abstract}
\maketitle
\section{Introduction and statement of the main results}
     In this work we will study the solvability of the following linear problem,
\begin{equation}\label{eq:prob}
\left\{
\begin{array}{rcl}
 u_t+(-\Delta)^{s} u&=&\l\dfrac{\,u}{|x|^{2s}}+ f\mbox{ in } \Omega\times (0,T),\\
u(x,t)&>&0\inn \Omega\times (0,T),\\
u(x,t)&=&0\inn (\ren\setminus\Omega)\times[ 0,T),\\
u(x,0)&=&u_0(x) \mbox{ if }x\in\O,
\end{array}
\right.
\end{equation}
and of the semilinear problem,
\begin{equation}\label{eq:probnonlinear}
\left\{
\begin{array}{rcl}
 u_t+(-\Delta)^{s} u&=&\l\dfrac{\,u}{|x|^{2s}}+ u^p+ f\mbox{ in } \Omega\times (0,T),\\
u(x,t)&>&0\inn \Omega\times (0,T),\\
u(x,t)&=&0\inn (\ren\setminus\Omega)\times[ 0,T),\\
u(x,0)&=&u_0(x) \mbox{ if }x\in\O,
\end{array}
\right.
\end{equation}
where $\Omega$ is a $C^{1,1}$ bounded domain in $\mathbb{R}^N$, $N> 2s$, $0<s<1$, $p>1$,  and $c$ and $\l$ are positive constants.
To avoid  the trivial case, we assume $0\in \Omega$.
We suppose  that $f$ and $u_0$ are non negative functions satisfying some hypotheses that we will precise later. By $(-\Delta)^s$ we denote the  fractional Laplacian of order $2s$, that is,
\begin{equation}\label{fraccionario}
(-\Delta)^{s}u(x):=a_{N,s}\mbox{ P.V. }\int_{\mathbb{R}^{N}}{\frac{u(x)-u(y)}{|x-y|^{N+2s}}\, dy},\, s\in(0,1),
\end{equation}
where
$$a_{N,s}:=2^{2s-1}\pi^{-\frac N2}\frac{\Gamma(\frac{N+2s}{2})}{|\Gamma(-s)|}$$
is the normalization constant so that the identity
$$(-\Delta)^{s}u=\mathcal{F}^{-1}(|\xi|^{2s}\mathcal{F}u),\, \xi\in\mathbb{R}^{N}, s\in(0,1),$$
holds for every $u\in \mathcal{S}(\mathbb{R}^N)$, the Schwartz class (see \cite{FLS}). This last identity justifies why we call fractional Laplacian to the integral operator. Notice that $(-\Delta)^{s}u$  is well defined if, for instance,
$u\in \mathcal{L}^{s}(\mathbb{R}^{N})\cap \mathcal{C}^{2s+\beta}_{\rm loc}(\mathbb{R}^{N})$ (or $\mathcal{C}^{1,2s+\beta-1}$ if $2s+\beta>1$), for some $\beta>0$.
Here
$$\mathcal{L}^{s}(\mathbb{R}^{N}):=\left\{
u:\mathbb{R}^{N}\to\mathbb{R}\quad \mbox{measurable}:\, \int_{\mathbb{R}^{N}}{\frac{|u(x)|}{1+|x|^{N+2s}}\, dx}<+\infty\right\},$$
endowed with the norm
$$\|u\|_{\mathcal{L}^{s}(\mathbb{R}^{N})}:=\int_{\mathbb{R}^{N}}{
\frac{|u(x)|}{ 1+|x|^{N+2s}}\, dx}.$$
Notice also that, for $u$ under these hypotheses, $(-\Delta)^{s}u$ is a continuous function. See \cite{S}.

In the case $s=1$, these problems correspond to the classical heat equation, and they have been deeply understood in the past years (see for instance
\cite{BaGo} for \eqref{eq:prob} and \cite{APP} for \eqref{eq:probnonlinear}). For $s\in (0,1)$, the fractional setting, there exists also a large  literature
dealing with the case $\l=0$. We refer for instance to \cite{CF, FK, IS} and the references therein.
A result on the uniqueness of positive regular solution to the linear problem can be found in \cite{BPSV}.

However, the case under consideration here, $\l>0$ and $s\in (0,1)$, is quite different; for instance,  any positive supersolution to problem \eqref{eq:prob} is
unbounded close to the origin, even for nice data. This fact, among other results, was proved in the local case by Baras-Goldstein in \cite{BaGo}. In the nonlocal framework, the precise rate of growth of the
solutions near the origin will be the key to obtain the optimal results. It is worthy to point out here the difference with the local case, where this rate is obtained just by solving an elementary linear differential equation.

The main results in this work can be summarized as follows.

First, to study the  local behavior of the solutions we need some sharp local estimates, that are based on a  \emph{Harnack inequality} for a related
problem resulting from the \emph{ground state transformation} by Frank, Lieb and Seiringer (\cite{FLS}), which is a problem with \emph{singular coefficients}.
This weak Harnack inequality gives the exact blow
up rate for the positive supersolutions near the $t$ axis. For the proof of this result, we closely follow the work of
Felsinger and Kassmann in \cite{FK}, where the authors develop a weak parabolic Harnack inequality for a general type of
nonlocal operators. This result does not apply straightforward to our singular operator, so, based on their scheme, we need
to check every step in the Moser's protocol (see \cite{Mos}). The Harnack inequality for singular weights can be useful in related problems.

The results obtained in this work for the linear problem \eqref{eq:prob} can be seen as the extension to the fractional setting
of those for the heat equation developed by P. Baras and J. A. Goldstein in \cite{BaGo}.
Nevertheless the proofs that we present are significantly different. More precisely, we obtain the optimal summability required to
the data in order to solve the problem and to prove the instantaneous
and complete blow up for $\lambda>\Lambda_{N,s}$  by using the results of Section \ref{HH}.
As a byproduct, we also prove the optimality of the power in the Hardy potential term, i.e., that $p>1$ is a \emph{supercritical power} for $\dfrac{u^p}{|x|^{2s}}$.
This result in the local framework was obtained by Brezis and Cabr\'{e} in
\cite{BC}.

Secondly, concerning the semilinear problem \eqref{eq:probnonlinear},  the main result that we obtain in this paper is that for all
$0<\lambda< \Lambda_{N,s}$  there exists a threshold exponent, $p_+(\lambda,s)$, for the existence of  positive solutions.
By \textit{threshold} we mean that when we consider an exponent  $p>p_+(\lambda,s)$,  there are no positive supersolutions even in the weak sense,
 while if $p<p_+(s,\lambda)$, it is possible to establish a suitable class of nonnegative data for which we can find a positive solution. We will see in particular that the threshold  exponent $p_+(s,\lambda)$ is the same as in the elliptic case (see \cite{BMP, F}).
In fact, this critical power is related to the possibility of finding a supersolution to the elliptic problem in the whole $\mathbb{R}^N$.
As  in the linear problem the main ingredient is the local  estimates of the solutions close to the origin.

The paper is organized as follows.

In Section \ref{2} we  describe the natural functional framework associated to the problems \eqref{eq:prob} and \eqref{eq:probnonlinear}. We define the two notions of solution we will use along the paper:  {\it weak} solutions and {\it energy} solutions. Moreover, we prove some comparison principles which are interesting themselves.

In Section \ref{3} we describe the radial solutions of the corresponding homogeneous elliptic problem in $\mathbb{R}^N$. These solutions will allow
us to precise the singularity of the supersolutions to problem \eqref{eq:prob} near the origin. This is a key point to perform
the so called \emph{ground state transformation} (see \cite{FLS}), and to obtain the nonexistence results afterwards.

Section \ref{HH} is devoted to obtain the weak Harnack inequality for the positive supersolutions of the problem resulting from the
\emph{ground state transformation} by Frank, Lieb and Seiringer, that introduces the difficulty of dealing with a kernel with \emph{singular coefficients}.

The goal of Section \ref{4} is to study the linear problem \eqref{eq:prob} and some consequences.

Finally, in Section \ref{5} we consider  the semilinear   problem
\eqref{eq:probnonlinear}. Furthermore, we prove a instantaneous and complete blow-up phenomenon.

We also include two appendices with  auxiliary results that seem to be useful in other applications.

Appendix A includes a H\"{o}lder
regularity result, that can be seen as the translation to the bounded domain
case  of some results in \cite{CF}. More precisely, we will prove that an energy solution is in fact a viscosity solution, and we can apply regularity results that ensure that it is indeed a \emph{strong solution} in the sense of Definition 1.3 in \cite{BPSV}.

In Appendix B we include some inequalities needed to prove the Harnack inequality in Section \ref{HH}.
As far as we know, these results involve some significative changes  respect to the standard ones, and therefore we consider them appropriate to be included here.

\section{Functional Framework: Some preliminary results}\label{2}

Along this paper we will always assume that $\Omega$ is a $C^{1,1}$ bounded domain of $\R^N$ with the origin inside.
We consider the fractional Sobolev space $H^{s}_0(\Omega)$ defined as
$$H^{s}_0(\Omega):=\{u\in H^s(\R^N) \hbox{ with } u=0 \hbox{ a.e. in } \R^N\setminus \Omega\},$$
endowed with the norm
$${\|u\|_{H^{s}_0(\Omega)}:=\left(\int_Q{\frac{|u(x)-u(y)|^2}{|x-y|^{N+2s}}\,dx\,dy}\right)^{1/2}},$$
where $Q=\R^{2N}\setminus(\mathcal{C}\Omega\times\mathcal{C}\Omega)$. The pair $(H^{s}_0(\Omega),\|\cdot\|_{H^{s}_0(\Omega)})$
yields a Hilbert space.
Moreover,
$$(-\Delta)^{s}:H^{s}_0(\Omega)\rightarrow H^{-s}(\Omega),$$
is a continuous operator, where $(-\Delta)^s$ is defined in \eqref{fraccionario}.

In what follows we will use the relation between the norm in the space $H_0^s(\Omega)$ and the $L^2$ norm of the fractional Laplacian, see \cite[Proposition~3.6]{DPV},
\begin{equation}\label{equivalencia}
\|u\|_{H^{s}_0(\Omega)}^2=2a_{N,s}^{-1}\|(-\Delta)^{s/2}u\|^2_{L^2(\R^N)}.
\end{equation}
It is easy to check that
for $u$ and $\varphi$ smooth enough, with vanishing conditions outside $\Omega$, we have the following duality product,
$$
2a_{N,s}^{-1}\int_{\mathbb{R}^{N}}{u(-\Delta)^s\varphi\, dx}=\int_{Q}{\frac{(u(x)-u(y))(\varphi(x)-\varphi(y))}{|x-y|^{N+2s}}\,dx\,dy},
$$
that in particular  implies the selfadjointness of $(-\Delta)^{s}$ in $H_0^s(\Omega)$.

We enunciate a Sobolev-type inequality that we will use throughout
the paper (see for example \cite{DPV} for a proof).
\begin{Theorem} {\it (Sobolev embedding).} \label{Sobolev}
Let $s\in(0,1)$. There exists a constant $S=S(N,s)$ such that, for
all $\phi\in \mathcal{C}^\infty_0(\ren)$, we have
$$\|\phi\|_{L^{2^{*}_{s}}(\mathbb{R}^{N})}^{2}\leq S\int_{\mathbb{R}^{N}}{\int_{\mathbb{R}^{N}}{\frac{|\phi(x)-\phi(y)|^{2}}{|x-y|^{N+2s}}\,dx\,dy}},$$
being
$$2^*_s=\frac{2N}{N-2s},$$
the so called fractional critical exponent.
\end{Theorem}

The parabolic problems studied in this article are related to the following Hardy inequality, proved in \cite{He} (see also \cite{B, FLS, SW, Y}).

\begin{Theorem}\label{DH}{\it (Fractional Hardy inequality).}
For all $u\in \mathcal{C}^{\infty}_{0}(\ren)$ the following inequality holds,
\begin{equation}\label{Hardy}
\dint_{\ren} \,|\xi|^{2s} |\hat{u}|^2\,d\xi\geq
\Lambda_{N,s}\,\dint_{\ren} |x|^{-2s} u^2\,dx,
\end{equation}
where
\begin{equation}\label{bestC}
\Lambda_{N,s}= 2^{2s}\dfrac{\Gamma^2(\frac{N+2s}{4})}{\Gamma^2(\frac{N-2s}{4})}.
\end{equation}
The constant $\Lambda_{N,s}$ is optimal and not attained.
\end{Theorem}

\begin{remark}

\

\begin{enumerate}
\item It can be checked that
$$\Lambda_{N,s}\to \Lambda_{N,1}:=\left(\dfrac{N-2}{2}\right)^2,$$
the classical Hardy constant, when $s$ tends to $1$. Moreover, by scaling it can be proved that the optimal constant is the same for every domain containing the pole of the Hardy potential.
\item The optimal constant defined in \eqref{bestC} coincides for every  bounded domain $\Omega$ containing the pole of the Hardy
potential. That is, if $0\in \Omega$, using \eqref{equivalencia} we can rewrite the Hardy inequality \eqref{Hardy} as
\begin{equation}\label{hardy}
\frac{a_{N,s}}{2}\int_{Q}{\frac{|u(x)-u(y)|^2}{|x-y|^{N+2s}}}\, dx\, dy\geq
\Lambda_{N,s}\int_{\Omega}{\frac{u^2}{|x|^{2s}}\,dx},\,u\in H_{0}^{s}(\Omega).
\end{equation}
The optimality of $\Lambda_{N,s}$ here follows by a scaling argument.
\end{enumerate}
\end{remark}

Consider the parabolic problem
$$
(P):=\left\{
\begin{array}{rcl}
u_t+(-\Delta)^{s} u&=& f(x,t,u)\mbox{ in } \Omega\times (0,T),\\
u(x,t)&=&0\inn (\ren\setminus\Omega)\times[ 0,T),\\
u(x,0)&=&u_0(x) \mbox{ if }x\in\O,
\end{array}
\right.
$$
Denote
\begin{equation*}\begin{split}
\mathcal{T}:=\{&\phi:\mathbb{R}^N\times [0,T]\rightarrow\mathbb{R},\,\hbox{ s.t. }-\phi_t+(-\Delta)^s\phi=\varphi,\,
\varphi\in L^\infty(\Omega\times (0,T))\cap \mathcal{C}^{\alpha, \beta}(\Omega\times (0,T)),\\
&\phi=0\inn (\ren\setminus \Omega)\times   {(0,T]},\phi(x,T)=0 \inn \Omega \}.
\end{split}\end{equation*}
Notice that every $\phi\in \mathcal{T}$ belongs in particular to $L^\infty(\Omega\times (0,T))$ (see \cite{LPPS}).
Moreover according with the results in Appendix A,  $\phi\in \mathcal{T}$ is a strong solution (see Definition 1.3 in \cite{BPSV}).

 We define the meaning of \emph{weak solution}.

\begin{Definition}\label{veryweak}  Assume $u_{0}\in L^{1}(\Omega)$.
We  say that $u\in \mathcal{C}([0,T); {L}^{1}(\O))$,  is a weak
supersolution (subsolution) of problem $(P)$
if $f(x,t,u)\in L^1(\Omega\times [0,T))$,          {\mbox{$u\geq (\leq)\,0$} in \mbox{$(\mathbb{R}^N\setminus \Omega)\times [0,T)$}, $u(x,0)\geq (\leq)\, u_0(x)$ in $\Omega$,} and for all nonnegative $\phi\in \mathcal{T}$ we have that
\begin{equation}\label{eq:subsuper}
\dint_0^T\dint_{\Omega}\,-\phi_t\, u\,dxdt+\dint_0^T\dint_{         {\R^N}} u (-\Delta)^{s}\phi\,dx\,dt\ge (\le)\\
\dint_0^T\dint_{\Omega}\,f\phi\,dxdt +\int_\Omega{u_0(x)\phi(x,0)\,dx}.
\end{equation}
If $u$ is super and subsolution  then we say that $u$
is a  weak solution to $(P)$.
\end{Definition}
The weak solution will be considered to formulate the optimal nonexistence results. For existence results, we will consider the classical notion of
{\it finite energy solutions}.
\begin{Definition}\label{energysol} Assume $u_{0}(x)\in L^{2}(\Omega)$.
We say that $u \in L^{2}(0,T;H^{s}(\mathbb{R}^{N}))$ with $u_t\in L^{2}(0,T; H^{-s}(\Omega))$ is a finite energy supersolution (respectively subsolution) of (P) if $f(x,t,u)\in L^2 (0,T; H^{-s} (\Omega)) $ and it satisfies
\begin{equation*}\begin{split}\label{para-wn}
\int_0^T \int_{\Omega}
u_{t} \varphi \,dx\,dt &+ \frac{a_{N,s}}{2}\int_0^T
 \int_{Q} \frac{\big( u(x,t) - u (y,t) \big)\big(\varphi(x,t)- \varphi(y,t) \big) }{|x-y|^{N+2s}} dx dy\,dt\\
 &\ge (\le) \dint_0^T\dint_{\Omega}\,f\varphi\,dx\,dt,
\end{split}\end{equation*}
for any nonnegative $ \varphi \in L^{2} (0,T ; H_0^s (\Omega))$, $\varphi=0$ in $(\mathbb{R}^N\setminus\Omega)\times (0,T)$.

If $u$ is super and subsolution then $u$ is a  finite energy solution.
\end{Definition}

\begin{remark}
If $u \in L^{2}(0,T;H_0^{s}(\Omega))$ and $u_t\in L^{2}(0,T; H^{-s}(\O))$, then by approximating with smooth functions and taking advantage of the hilbertian structure of the space, it can be checked that $u\in \mathcal{C}([0,T]; L^{2}(\Omega))$.

Notice that both definitions can be considered, by scaling, in $\mathbb{R}^{N}\times [T_1, T_2)$ with $[T_{1}, T_{2})\subset [0,T)$.
\end{remark}
The existence and uniqueness of an energy solution to problem (P) when $F$ is in the dual space $L^2(0,T;H^{-s}(\O))$ can be obtained by means of a direct Hilbert space approach.
See the result by A. N. Milgram in \cite{MI} based on a method of Vishik in \cite{V}, that is essentially an extension of the Lax-Milgram
theorem to parabolic problems. More precisely, we have the following result.

\begin{Theorem}
Let $f\in L^2(0,T;H^{-s}(\O))$, then problem \eqref{eq:prob} has a unique finite energy solution.
\end{Theorem}
See \cite[Theorem 26]{LPPS} for a detailed proof in this fractional framework.
\begin{remark}\label{coerciveoperator}
Notice that by defining
\begin{eqnarray}\label{Lhardy}
L_{\phi}(u)&:=& \int_0^T \int_{\Omega}
-u \phi_{t} \,dx\,dt + \frac{a_{N,s}}{2}\int_0^T
 \int_{Q} \frac{\big( u(x,t) - u (y,t) \big)\big(\phi(x,t)- \phi(y,t) \big) }{|x-y|^{N+2s}} dx dy\,dt\nonumber\\
 &&-\lambda\int_0^T\int_\Omega{\frac{u\phi}{|x|^{2s}}\,dx\,dt},
\end{eqnarray}
and
$$
\bra \varphi, \phi\ket _{*}=\frac{1}{2} \bra \varphi(x,0), \phi(x,0)\ket_{L^{2}(\Omega)}+\frac{a_{N,s}}{2}\left(1-\dfrac{\lambda}{\Lambda_{N,s}}\right)\bra\varphi, \phi\ket_{L^{2}(0,T; H_{0}^{s}(\Omega))},
$$
thanks to the Hardy inequality (see \eqref{Hardy}) one can reproduce the proof of \cite[Theorem 26]{LPPS} to assure the existence and uniqueness of an energy solution to the problem
$$
(P_\lambda):=\left\{
\begin{array}{rcl}
u_t+(-\Delta)^{s} u-\lambda\dfrac{u}{|x|^{2s}}&=& F(x,t)\mbox{ in } \Omega\times (0,T),\\
u(x,t)&=&0\inn (\ren\setminus\Omega)\times[ 0,T),\\
u(x,0)&=&u_0(x) \mbox{ if }x\in\O,
\end{array}
\right.
$$
for $F\in L^2(0,T;H^{-s}(\O))$, $u_0\in L^2(\Omega)$, and $\lambda<\Lambda_{N,s}$. For the case $\lambda=\Lambda_{N,s}$, consider  the
Hilbert space $H(\O)$ defined as the completion of
$\mathcal{C}^{\infty}_0(\O)$ with respect to the norm
\begin{equation}\label{Hardynorm}
\|u\|^2_{H(\Omega)}:=\frac{a_{N,s}}{2}\|u\|^2_{H_0^s(\Omega)}-\Lambda_{N,s}\int_{\Omega}{\frac{u^2}{|x|^{2s}}}\,dx.
\end{equation}
In \cite{Frank}, the author proves the following improved Hardy
inequality,
\begin{equation}\label{improvedHardy}
\frac{a_{N,s}}{2}\|u\|^2_{H_0^s(\Omega)}-\Lambda_{N,s}\int_{\Omega}{\frac{u^2}{|x|^{2s}}\,dx}\geq C (\O,q, N,s)\|u\|^2_{W^{\tau, 2}_0(\Omega)},
\end{equation}
for all $s/2<\tau<s$ (see also \cite{F, APP2} for
alternative proofs without using the Fourier transform). Thus we
can see that $H(\Omega)\subset W^{\tau, 2}_0(\Omega)$ and
therefore, $H(\Omega)$ is compactly embedded in $L^p(\Omega)$ for
all $1\leq p < 2^*_s$ (see \cite[Corollary 7.2]{DPV}). Therefore, the
proof remains the same considering $L_\phi(u)$ as in
\eqref{Lhardy} (setting $\lambda=\Lambda_{N,s}$), and defining the
scalar product $\langle\cdot,\cdot\rangle_*$ as
$$
\bra \varphi, \phi\ket _{*}=\frac{1}{2} \bra \varphi(x,0), \phi(x,0)\ket_{L^{2}(\Omega)}+\bra\varphi, \phi\ket_{L^{2}(0,T; H^{s}_{0}(\Omega))},
$$
where the last term follows from \eqref{Hardynorm}.
\end{remark}

In order to study monotonicity approaches, we will need to prove comparison results for both kind of solutions.

\begin{Lemma}{\it (Weak Comparison Principle).}\label{DebilComparison}
Let $0\leq\lambda\leq\Lambda_{N,s}$ and let $u,\, v\in \mathcal{C}([T_1,T_2);L^1(\Omega))$ be weak solutions to the problems
$$\begin{array}{ll}
\begin{cases}u_t+(-\Delta)^s u-\lambda\dfrac{u}{|x|^{2s}}=f_1 \hbox{ in }\Omega\times (T_1,T_2), \\
u=g_1\inn (\mathbb{R}^{N}\setminus\Omega)\times[ T_1,T_2),\\
u(x,T_1)=h_1(x)\inn\Omega,
\end{cases}

&

\begin{cases}v_t+(-\Delta)^s v-\lambda\dfrac{v}{|x|^{2s}}=f_2 \hbox{ in }\Omega\times (T_1,T_2), \\
v=g_2\inn (\mathbb{R}^{N}\setminus\Omega)\times[T_1,T_2),\\
v(x,T_1)=h_2(x)\inn\Omega,
\end{cases}
\end{array}$$
respectively, where $f_1,\,f_2\in L^{1}(\Omega\times (T_1, T_2)),\, g_1, g_2\in L^{1}((\ren\setminus\Omega)\times (T_1, T_2))
\mbox{ and }h_{1}, h_2\in L^{1}(\Omega)$.

If $f_1\leq f_2$  in $\Omega\times (T_1, T_2)$, $g_1\leq g_2$ in $(\ren\setminus\Omega)\times [T_1, T_2)$ and $h_1\leq h_2$ in $\Omega$, then $u\leq v\inn \mathbb{R}^{N}\times (T_1,T_2)$.
\end{Lemma}

\begin{proof}
Define $w=v-u$. Hence, $w$ is a weak solution of
$$\begin{cases}w_t+(-\Delta)^s w-\lambda\dfrac{w}{|x|^{2s}}=f_2-f_1\ge 0 \hbox{ in }\Omega\times (T_1,T_2), \\
w=g_2-g_1\geq 0\inn (\mathbb{R}^{N}\setminus\Omega)\times[T_1,T_2),\\
w(x,T_1)=h_2-h_1\geq 0\inn\Omega.
\end{cases}$$
Consider now $\Phi\in \mathcal{C}_0^\infty(\Omega\times (T_1,T_2))$, $\Phi\ge 0$, and the solution $\varphi_{n}$ to the problem
\begin{equation}
\left\{\begin{array}{rcl}
-(\varphi_n)_t+(-\Delta)^s \varphi_n&=&\lambda\dfrac{\varphi_{n-1}}{|x|^{2s}+\frac{1}{n}}+\Phi \quad\mbox{in }\Omega\times (T_1,T_2),\\
\varphi_n&=&0\inn(\ren\setminus\Omega)\times[ T_1,T_2),\\
\varphi_n(x,T_2)&=&0\inn\Omega,
\end{array}\right.
\end{equation}
with
\begin{equation}
\left\{\begin{array}{rcl}
-(\varphi_0)_t+(-\Delta)^s \varphi_0&=&\Phi \quad\mbox{in }\Omega\times (T_1,T_2),\\
\varphi_0&=&0\inn(\ren\setminus\Omega)\times[T_1,T_2),\\
\varphi_0(x,T_2)&=&0\inn\Omega.
\end{array}\right.
\end{equation}

     Since $\varphi_n$ is regular in $\Omega\times  {[}T_1,T_2)$ and bounded in $\mathbb{R}^N\times (T_1,T_2)$
(see Appendix \ref{appreg}), this equation can be understood in a pointwise sense.
Moreover, by the Strong Comparison Principle, we know that $\varphi_n\ge 0$ and  $\varphi_{n-1}\leq \varphi_n$ in $\mathbb{R}^N\times {[}T_1,T_2)$ for all $n\in\mathbb{N}$.

     Hence, by the definition of weak solutions, and using that $w\geq 0$
\begin{eqnarray*}
\int_{T_1}^{T_2}\int_\Omega w\Phi\,dx\,dt&=&\int_{T_1}^{T_2}\int_\Omega w(-\varphi_n)_t\,dx\,dt+\int_{T_1}^{T_2}\int_\Omega w(-\Delta)^s \varphi_n\,dx\,dt-\lambda\int_{T_1}^{T_2}\int_\Omega \dfrac{w\varphi_{n-1}}{|x|^{2s}+\frac{1}{n}}\,dx\,dt\\
&\geq& \int_{T_1}^{T_2}\int_\Omega w(-\varphi_n)_t\,dx\,dt+\int_{T_1}^{T_2}\int_{\Omega} w(-\Delta)^s \varphi_n\,dx\,dt-\lambda\int_{T_1}^{T_2}\int_\Omega \dfrac{w\varphi_{n}}{|x|^{2s}}\,dx\,dt\\
&=&\int_{T_1}^{T_2}\int_\Omega (f_2-f_1)\varphi_n\,dx\,dt +\int_\Omega{w(x,T_1)\varphi_n(x,T_1)\,dx} \ge 0,
\end{eqnarray*}
for all $\Phi\in \mathcal{C}_0^\infty(\Omega\times (T_1,T_2))$, $\Phi\ge 0$. Thus, $w\ge 0$ in $\mathbb{R}^N\times (T_1,T_2)$, and therefore $u\leq v\inn \mathbb{R}^{N}\times (T_1,T_2)$.
\end{proof}

\begin{Corollary}{\it (Uniqueness of weak solutions for the linear problem).}
\\Let suppose $F\in L^1(\Omega\times(0,T))$. Then problem $(P_\lambda)$ has at most one nontrivial weak solution.
\end{Corollary}

The comparison result for energy solutions can be proved in a standard way, so we skip the proof (see for example \cite{BMP} for a proof in the elliptic case).
\begin{Lemma}{\it Energy Comparison Principle.}\label{Comparison Principle}
Let $0\leq\lambda<\Lambda_{N,s}$ and let $u,\, v\in L^{2}(T_1,T_2;
         {H^{s}(\ren)})$ with $u_t,\,v_t\in L^{2}(T_1,T_2; H^{-s}(\O))$ be finite energy solutions to the problems
$$\begin{array}{ll}
\begin{cases}u_t+(-\Delta)^s u-\lambda\dfrac{u}{|x|^{2s}}=f_1 \hbox{ in }\Omega\times (T_1,T_2), \\
u=g_1\inn (\mathbb{R}^{N}\setminus\Omega)\times[ T_1,T_2),\\
u(x,T_1)=h_1(x)\inn\Omega,
\end{cases}
&
\begin{cases}v_t+(-\Delta)^s v-\lambda\dfrac{v}{|x|^{2s}}=f_2 \hbox{ in }\Omega\times (T_1,T_2), \\
v=g_2\inn (\mathbb{R}^{N}\setminus\Omega)\times[T_1,T_2),\\
v(x,T_1)=h_2(x)\inn\Omega,
\end{cases}
\end{array}$$
respectively, where $f_1, f_2\in L^2(T_1, T_2; H^{-s}(\Omega))$, $g_1, g_2\in L^{2}(T_1, T_2; L^{2}(\ren\setminus\Omega))$ and
 $h_1, h_2\in L^{2}(\Omega)$. If $f_1\leq f_2 \inn \Omega\times (T_1,T_2)$, $g_1\leq g_2\inn (\mathbb{R}^{N}\setminus\Omega)\times[T_1,T_2) $ and $h_1\leq h_2\inn\Omega$, then $u\leq v\inn \mathbb{R}^{N}\times (T_1,T_2)$.
\end{Lemma}
\begin{remark}
Notice that if $\lambda=\Lambda_{N,s}$, we can obtain the same result for $u,v\in L^2(T_1,T_2;H(\Omega))$, where $H(\Omega)$ was defined in \eqref{Hardynorm}, only by exactly repeating this proof.
\end{remark}

Finally, consider the problem
\begin{equation}\label{hom}
\left\{
\begin{array}{rcl}
u_t+(-\Delta)^{s} u&=&0\mbox{ in } \Omega\times (0,T),\\
u(x,t)&=&0\inn (\ren\setminus\Omega)\times[ 0,T),\\
u(x,0)&=&u_0(x)\gneq 0 \mbox{ if }x\in\O.
\end{array}
\right.
\end{equation}
We enunciate a \emph{Weak Harnack Inequality} that we will use along the paper (see \cite[Theorem 1.1]{FK} even in a more general setting).

\begin{Lemma}{\it (Weak Harnack Inequality). }\label{Harnack inequality}
If $u$ is a non negative supersolution of \eqref{hom} in $\Omega\times (0,T)$, then          {for every $t_0\in (0,T)$} there exists $r>0$ and a positive constant $C=
C(N,s, r, t_0,\beta)$ such that
$$
\iint_{R^-}u(x,t)\,dx\,dt\leq C\big(\mbox{ess}
\inf\limits_{R^+}u\big),
$$
where $R^-=B_{r}(0)\times (t_0-\frac{3}{4}\beta,
t_0-\frac{1}{4}\beta) $, $R^{+}=B_{r}(0)\times
(t_0+\frac{1}{4}\beta, t_0+\frac{3}{4}\beta)$.
\end{Lemma}

As a consequence of this lemma, we can formulate the {\it strong maximum principle}.
\begin{Theorem}{\it (Strong Maximum Principle).}\label{SMP}
If $u$ is a non negative supersolution of \eqref{hom}, then $u(x,t)>0$ in $\Omega\times (0,T)$.
\end{Theorem}

\section{Local behavior of solutions of the stationary equation}\label{3}
     The purpose of this section is to analyze the behavior of the radial solutions to the homogeneous problem
\begin{equation}\label{problemHardy}
(-\Delta)^{s} u=\l\dfrac{\,u}{|x|^{2s}} \mbox{ in } \mathbb{R}^N\setminus\{0\}
\end{equation}
in a neighborhood of the origin, in order to use this information as a tool for proving the existence and nonexistence results.

\begin{Lemma} \label{singularity} Let $0<\lambda\leq \Lambda_{N,s}$. Then $v_{\pm\alpha}=|x|^{-\frac{N-2s}{2}\pm\alpha}$ are
solutions to
\begin{equation}\label{homogeneous}
(-\Delta)^s u= \lambda\frac{u}{|x|^{2s}}\inn (\ren\setminus{\{0\}}),
\end{equation}
where $\alpha$ is obtained by the identity
\begin{equation}\label{lambda}
\lambda=\lambda(\alpha)=\lambda(-\alpha)=\dfrac{2^{2s}\,\Gamma(\frac{N+2s+2\alpha}{4})\Gamma(\frac{N+2s-2\alpha}{4})}{\Gamma(\frac{N-2s+2\alpha}{4})\Gamma(\frac{N-2s-2\alpha}{4})}.
\end{equation}
\end{Lemma}
\begin{proof}
Applying the Fourier transform of radial functions (see for instance \cite[Theorem 4.1]{SW}) it yields,

\begin{eqnarray*}
\mathcal{F} (v_{\alpha})(\xi)&=& \xi^{-\frac{N-1}{2}}\dint\limits_{0}^{\infty} (r\xi)^{\frac{1}{2}} J_{\frac{N-2}{2}}(r\xi)v_{\alpha}(r)\,r^{\frac{N-1}{2}}\,dr=\,\xi^{-\frac{N}{2}-s-\alpha}\dint_{0}^{\infty} (r\xi)^{s+\alpha} J_{\frac{N-2}{2}}(r\xi)\,d(r\xi)\\
&=&2^{\alpha+s}\dfrac{\Gamma(\frac{N+2s+2\alpha}{4})}{\Gamma(\frac{N-2s-2\alpha}{4})}\xi^{-\frac{N}{2}-s-\alpha},
\end{eqnarray*}
where $J_{\frac{N-2}{2}}$ denotes the Bessel function of the first kind
\begin{equation*}
  J_{\nu }(t)=\bigg(\frac{t}{2}\bigg)^{\!\!\nu }\sum\limits_{k=0}^{\infty }
  \dfrac{(-1)^{k}}{\Gamma (k+1)\Gamma (k+\nu +1)}\bigg(\frac{t}{2}\bigg)
  ^{\!\!2k}.
\end{equation*}

     Now, we notice that
$$
(-\Delta )^{s} v_{\alpha}=\mathcal F^{-1}(         {|\xi|}^{2s} \mathcal{F} (v_{\alpha})(\xi))= 2^{\alpha+s}\dfrac{\Gamma(\frac{N+2s+2\alpha}{4})}{\Gamma(\frac{N-2s-2\alpha}{4})} \mathcal{F}^{-1} (         {|\xi|}^{-\frac{N}{2}+s-\alpha})= \lambda |x|^{-2s} v_{\alpha}
$$
with $\lambda=\lambda(\alpha)$ equal to \eqref{lambda}.
\end{proof}
\begin{remark}
Notice that $\lambda(\alpha)= \lambda(-\alpha)=m_{\alpha}m_{-\alpha}$, with $m_{\alpha}= 2^{\alpha+s}\dfrac{\Gamma(\frac{N+2s+2\alpha}{4})}{\Gamma(\frac{N-2s-2\alpha}{4})}$.
\end{remark}

\begin{Lemma}\label{LambdaVsSing}
The following equivalence holds true:
$$
0<\lambda(\alpha)=\lambda(-\alpha)\leq \Lambda_{N,s}\mbox{ if and only if } 0\leq \alpha<\dfrac{N-2s}{2}.
$$
\end{Lemma}
For the reader convenience, we include an elemental proof of this Lemma (see also \cite{FLS, He}).
\begin{proof}
Notice that $\lambda(\alpha)$ is a positive continuous function for $0\leq \alpha<\dfrac{N-2s}{2}$, such that $\lambda(0)=\Lambda_{N,s}$. It is sufficient to prove that for fixed $s$, $\lambda(\alpha)$ is a decreasing function.

     Let consider the following representation of the Gamma function (see \cite{A} for more details):
$$
\dfrac{1}{\Gamma(x)}= x e^{\gamma x}\prod\limits_{n=1} ^{\infty} \big(1+\frac{x}{n}\big) e^{-\frac{x}{n}},
$$
where $\gamma$ is the Euler-Mascheroni constant (see for instance \cite{AS}). We aim to prove that $\log \dfrac{1}{\lambda(\alpha)}$ is an increasing function in $\alpha$.
$$
\begin{array}{rcl}
\log \dfrac{1}{\lambda(\alpha)}&=& \log \dfrac{1}{2^{2s}}\dfrac{\frac{1}{\Gamma (\frac{N+2s+2\alpha}{4})}\cdot \frac{1}{\Gamma(\frac{N+2s-2\alpha}{4})}}{\frac{1}{\Gamma(\frac{N-2s+2\alpha}{4})}\cdot\frac{1}{\Gamma(\frac{N-2s-2\alpha}{4})}}\\
\\
&=&-2s\cdot\log 2+\log \dfrac{(N+2s)^2-4\alpha^2}{(N-2s)^2-4\alpha^2}+2\gamma s+\sum\limits_{n=1}^{\infty} \Big[ \log\dfrac{(\frac{N+4n+2s}{4n})^2-\frac{\alpha^2}{4n^2}}{(\frac{N+4n-2s}{4n})^2-\frac{\alpha^2}{4n^2}}-\dfrac{2s}{n}\Big].
\end{array}
$$
Notice that the last term is a convergent series in the same way as
$$\prod\limits_{n=1} ^{\infty} \big(1+\frac{x}{n}\big) e^{-\frac{x}{n}}$$
is a convergent product.
We conclude just by noticing that if $a>b$ and $\zeta>0$, then $\dfrac{a^2-\zeta^2}{b^2-\zeta^2}$ is an increasing function in $\zeta$.
\end{proof}
\begin{remark}\label{gamma1}
Notice that we can explicitly construct two positive solutions to the homogeneous problem \eqref{homogeneous}. Henceforth, we denote
\begin{equation}\label{g1}
\gamma= \dfrac{N-2s}{2}-\alpha \hbox{ and } \bar\gamma= \dfrac{N-2s}{2}+\alpha,
\end{equation}
with $0<\gamma\leq \dfrac{N-2s}{2}\leq\bar\gamma<(N-2s)$. Since $N-2\gamma-2s={2}\alpha>0$ and $N-2\bar\gamma-2s=-{2}\alpha<0$, then {$(-\Delta)^{s/2}(|x|^{-\gamma})\in L^2(\Omega)$, but $(-\Delta)^{s/2}(|x|^{-\bar\gamma})$ does not.}
\end{remark}

We will use these results to study the unboundedness of any positive weak supersolution, and moreover, to obtain an explicit quantitative information on the growth around the origin when the summability of the datum is good enough.

\section{Weak Harnack inequality for a weighted problem }\label{HH}
Frank, Lieb and Seiringer proved in \cite[Proposition 4.1]{FLS} the following representation result.
\begin{Lemma}\label{GSrepresentation} (Ground State Representation)
Let $0<\gamma< \frac{N-2s}{2}$. If $\phi\in C_{0}^{\infty}
(\mathbb{R}^{N})$ and $\bar{\phi}(x):= |x|^{\gamma} \phi(x)$, then
\begin{equation}\label{GroundState}
\dint_{\mathbb{R}^{N}} \, |\xi|^{2s} |\hat{\phi}(\xi)|^{2}\,d\xi
-(\Lambda_{N,s}+\Phi_{N,s}(\gamma))
\dint_{\mathbb{R}^{N}}|x|^{-2s}|\phi(x)|^2\,dx=  {\frac{a_{N,s}}{2}}\dint\dint_{\mathbb{R}^{2N}}\,
\frac{|\bar{\phi}(x)-\bar{\phi}(y)|^2}{|x-y|^{N+2s}}\frac{dx}{|x|^{\gamma}}\,
\frac{dy}{|y|^{\gamma}},
\end{equation}
where
\begin{equation}\label{Phi}
\Phi_{N,s}(\gamma):=2^{2 s}\left(\frac{\Gamma\big(\frac{\gamma+2s}{2}\big)\Gamma\big(\frac{N-\gamma}{2}\big)}{\Gamma\big(\frac{N-\gamma-2s}{2}\big)\Gamma\big(\frac{\gamma}{2}\big)}-
\frac{\Gamma^2\big(\frac{N+2s}{4}\big)}{\Gamma^2\big(\frac{N-2s}{4}\big)}\right).
\end{equation}
\end{Lemma}
A relevant fact for us is the following result.
\begin{Proposition}\label{psi}
Consider the function
$$
\begin{array}{rcl}
\Psi_{N,s}:[0, \frac{N-2s}{2}]&\rightarrow& [0,\Lambda_{N,s}]\\
&&\\
\gamma &\rightarrow& \Psi_{N,s}(\gamma):=\Lambda_{N,s}+\Phi_{N,s}(\gamma),
\end{array}
$$
where $\Phi_{N,s}$ is defined by \eqref{Phi}. Then $\Psi_{N,s}$  is strictly increasing and surjective.

\end{Proposition}
Notice that         {, considering $\gamma$ defined in \eqref{g1},} $\lambda(\alpha)=\Psi_{         {N,s}}\left(\dfrac{N-2s}{2}-\alpha\right)$, and therefore
for any $0<\lambda<\Lambda_{N,s}$, there exists $\alpha\in          {(}0, \frac{N-2s}{2}         {)}$, such that
$$\lambda=\lambda(\alpha)=\dfrac{2^{2s}\,\Gamma(\frac{N+2s+2\alpha}{4})\Gamma(\frac{N+2s-2\alpha}{4})}{\Gamma(\frac{N-2s+2\alpha}{4})\Gamma(\frac{N-2s-2\alpha}{4})}.$$
Taking $0<\gamma=\frac{N-2s}{2}-\alpha< \frac{N-2s}{2}$, and
$\Lambda_{N,s}+\Phi_{N,s}(\gamma)=\lambda(\alpha)$, by Lemma \ref{GSrepresentation} we can write
the energy as
\begin{equation}\label{energyGS}
\dint_{\mathbb{R}^{N}} \, |\xi|^{2s} |\hat{u}(\xi,t)|^{2}\,d\xi
-\lambda(\alpha)\dint_{\mathbb{R}^{N}}|x|^{-2s}|u(x,t)|^2\,dx=  {\frac{a_{N,s}}{2}}\dint\dint_{\mathbb{R}^{2N}}\,
\frac{|v(x,t)-v(y,t)|^2}{|x-y|^{N+2s}}\frac{dx}{|x|^{\gamma}}\,
\frac{dy}{|y|^{\gamma}},
\end{equation}
where $v(x,t):=|x|^\gamma u(x,t)$. The Euler-Lagrange equation associated to this identity is
$$
(-\Delta)^{s} u-\l\dfrac{\,u}{|x|^{2s}}=|x|^\g L_\g v(x,t),
$$
where
$$
L_\g (v(x,t)):= a_{N,s}\:\:P.V. \int_{\mathbb{R}^{N}}
(v(x,t)-v(y,t))K(x,y)dy,
$$
and
$$K(x,y)=\dfrac{1}{|x|^\g}\dfrac{1}{|y|^\g}\dfrac{1}{|x-y|^{N+2s}},\quad 0<\g=\frac{N-2s}{2}-\alpha<\frac{N-2s}{2}  .$$
Thus we conclude that if $u$  is an energy solution of problem \eqref{eq:prob}  with $0<\lambda<\Lambda_{N,s}$, then
 $v$ solves the parabolic equation
\begin{equation}\label{peso1}
\left\{
\begin{array}{rcl}
 |x|^{-2\g}v_t+L_\g v&=&|x|^{-\g}f(x,t)\mbox{ in } \Omega\times (0,T),\\
v(x,t)&=&0\inn (\ren\setminus\Omega)\times[ 0,T),\\
v(x,0)&=&v_0(x):= |x|^\g u_0(x)\mbox{ if }x\in\O.
\end{array}
\right. \end{equation}
         {Therefore, if we want to analyze the behavior of $u$ near the origin, we may, equivalently, deal with the same question for $v$.} In particular, we will prove that the  weighted operator
$$|x|^{-2\g}v_t-L_\g v,$$
 satisfies a suitable weak Harnack inequality. In the local case, this kind of result can be obtained as a consequence of some results by Chiarenza-Frasca, Chiarenza-Serapioni and Gutierrez-Wheden, see \cite{CFr, Cs, GW} and the references therein.

Let us precise first the natural functional framework associated to the new problem \eqref{peso1}. For simplicity of typing we  denote
\begin{equation}\label{defMuNu}
 d\mu:=\dfrac{dx}{|x|^{2\g}},\quad \hbox{ and }
\quad d\nu:= K(x,y)dxdy.
\end{equation}
Let $\O\subseteq \ren$. We define the weighted Sobolev space $Y^{s,\g}(\O)$ as
$$
Y^{s,\g}(\O)\dyle := \Big\{ \phi\in L^2(\O,  d\mu):\dint_{\O}\dint_{\O}(\phi(x)-\phi(y))^2d\nu<+\infty\Big\}.
$$
It is clear that $Y^{s,\g}(\O)$ is a Hilbert space endowed with the
norm
$$
\|\phi\|_{  {Y^{s,\g}(\O)}}:=
\Big(\dint_{\O}|\phi(x)|^2d\mu
+\dint_{\O}\dint_{\O}(\phi(x)-\phi(y))^2d\nu\Big)^{\frac 12},
$$
and we define the space $Y^{s,\g}_0(\O)$ as the
completion of $\mathcal{C}^\infty_0(\O)$ with respect to this norm. In particular, we denote
$$
|||\phi|||_{Y^{s,\g}_0(\O)}:=
\Big(\dint_{\O}\dint_{\O}(\phi(x)-\phi(y)|^2d\nu\Big)^{\frac 12}.
$$
If $\O$ is bounded, the norms $|||\cdot|||_{Y^{s,\g}_0(\O)}$ and $ \|\cdot\|_{  {Y^{s,\g}(\O)}}$ are equivalent (see Theorem \ref{poincare} in Appendix B for more details).
If $\Omega=\ren$, using the definition of
$L_{\g}$, we obtain that for all $w_1,w_2\in Y^{s,\g}_0(\ren)$,
$$
\langle L_{\g}(w_1),w_2\rangle
=  \frac{a_{N,s}}{2}\dint_{\ren}\dint_{\ren}(w_1(x)-w_1(y))(w_2(x)-w_2(y))d\nu.
$$
Let us begin by the following natural definition.
\begin{Definition}\label{super}
Let $v \in \mathcal{C}(  {[0,T)}, L^2(\ren, d\mu))\cap
L^{2}(0,T;Y^{s,\g}(\ren))$. We say that $v$ is a
supersolution to problem \eqref{peso1} if $v(x,0)\ge v_0(x)$ and
for all $\O_1\subset\subset \O$, for all $[t_1,t_2]\subset (0,T)$
we have
\begin{equation*}
\begin{split}\label{para-wnn}
  {\int_{t_1}^{t_2}\int_{\Omega_1}{-\varphi_t v\,d\mu         {dt}}}&   {+\frac{a_{N,s}}{2}         {\int_{t_1}^{t_2}}\int_{\tilde{Q}}( v(x,t) - v (y,t))
(\varphi(x,t)- \varphi(y,t))d\nu\,dt}\\
 &  {\ge \dint_{t_1}^{t_2}\dint_{\Omega_1}\,f\varphi\,|x|^{-\gamma}dx\,dt+
 \int_{\Omega_1} \varphi(x,t_1)v(x,t_1)d\mu-\int_{\Omega_1}
\varphi(x,t_2)v(x,t_2)d\mu,}
\end{split}\end{equation*}
for any nonnegative $ \varphi \in L^{2} (t_1,t_2;
Y^{s,\g}_0(\Omega_1))$, such that $\varphi_t\in L^2(t_1,t_2;Y^{-s,\gamma}(\Omega_1))$,  where $\tilde{Q}:=\mathbb{R}^{2N}\setminus(\mathcal{C}\Omega_1\times\mathcal{C}\Omega_1)$.
\end{Definition}

The main  result of this section is the next Theorem.
\begin{Theorem}{\it(Weak Harnack Inequality)}\label{harnack}
Assume that $f$ (resp. $v_0$) $\ge 0$ in $\O\times (0,T)$ (resp.
in $\O$). Let $v \in \mathcal{C}(  {[0,T)}, L^2(\ren, d\mu))\cap
L^{2}(0,T;Y^{s,\g}(\ren))$ be  a
supersolution to \eqref{peso1} with $v\gneqq 0$ in $\ren\times
(0,T)$.

Then for any
$q<1+\frac{2s}{N}$, we have
\begin{equation}\label{main}
\Big(\int\!\!\!\int_{Q_1}v^qd\mu dt\Big)^{\frac 1q}\le C \inf_{Q_2}v,
\end{equation}
where $Q_1=B_r(x_0)\times (t_1,t_2), Q_2=B_r(x_0)\times (t_3,t_4)$
with $  {0<}t_1<t_2<t_3<t_4<T$, $  {B_r(x_0)\subset\Omega}$ and $C=C(N,r,t_1,t_2,t_3,t_4)>0$.
\end{Theorem}

The proof of this result follows  the classical arguments by Moser (see \cite{Mos}) with some necessary adaptations. In the context of the parabolic fractional-like operators the precedent work is the interesting  paper by Felsinger and Kassmann, \cite{FK}, that we will closely follow here, adapting the proofs to the weighted operator appearing in \eqref{peso1}.  To make the paper self-contained, we include here the corresponding proofs.

First of all, we will need an iteration result, originally proved
in \cite{Mos} and extended by Bombieri and Giusti in \cite{BG} to the case of general measures in the elliptic setting (see also  \cite[Lemma 2.2.6]{SC}).
\begin{Lemma}\label{clasi}
Let $\{U(r)\}_{\theta\le r\le 1}$ be a nondecreasing family of bounded
domains $U(r)\subset \mathbb{R}^{N+1}$,   {and let $m$, $c_0$ be positive constants, $\eta\in (0,1)$, $\theta\in [\frac 12,1]$ and $0<p_0\leq +\infty$}. Let $w$ be a positive,
measurable function defined on $U(1)$ satisfying
$$
\|w\|_{L^{p_{0}}(U(r), d\mu dt)}\le
\Big(\frac{c_0}{(R-r)^m  {|U(1)|_{d\mu\times dt}}}\Big)^{\Big(\frac{1}{p_0}-\frac{1}{p}\Big)}\Big(\int\!\!\!\!\int_{U(  {R})}w^{p}  {d\mu\,dt}\Big)^{\frac{1}{p}}
$$
  {for all} $r,R\in [\theta, 1]$,   {$r<R$},   {and for all} $p\in (0, \min\{\eta p_0, 1\})$.

     Assume also that
$$
\forall s>0:\; |U(1)\cap\{\log w>s\}|_{  {d\mu\times dt}} \le\frac{  {c_0}|U(1)|_{  {d\mu\times dt}}}{  {s}}.
$$
Then there exists $C(\theta,\eta,  {c_0,m},p_0)$ such that
$$
\Big(\int\!\!\!\!\int_{U(\theta)}w^{p_0}  {d\mu\,dt}\Big)^{\frac{1}{p_0}}\le
C|U(1)|_{  {d\mu\times dt}}^{\frac{1}{p_0}}.
$$
\end{Lemma}

     Hereafter, we will make use of the following notation. Given $r>0$, we define
\begin{equation}\label{defI}
I_-(r):=(-r^{2s},0),\qquad I_+(r):=(0,r^{2s}),
\end{equation}
\begin{equation}\label{defQ}
Q_-(r):= B_r(0)\times I_-(r),\qquad Q_+(r):=B_r(0)\times I_+(r).
\end{equation}

The first step to prove Theorem \ref{harnack} is to establish the
next estimate (see \cite[Proposition 3.4]{FK}). Notice that we
just have to consider the case where $B_r(x_0)=B_r(0)$. For
simplicity, we will write $B_r$ instead of $B_r(0)$.
\begin{Lemma}\label{negative}
Assume that $\frac 12\le r<R\le 1$ and let $p>0$.  Consider
$v\gvertneqq 0$ a supersolution to \eqref{peso1}, then
\begin{equation}\label{est1}
\Big(\int\!\!\!\int_{Q_-(r)}v^{-\tau p}d\mu dt\Big)^{\frac{1}{\tau}}\le
A \int\!\!\!\int_{Q_-(R)}v^{-p}d\mu dt,
\end{equation}
where $\tau:=1+\frac{2s}{N}$ and
$$
A:=A(  {N,s,p},r,R, \gamma)=C  {(N,s,\gamma)}(p+1)^2\left(\frac{1}{(R-r)^{2s}}+\frac{1}{R^{2s}-r^{2s}}\right)^\tau.
$$
\end{Lemma}
\begin{proof} Without loss of generality we can assume that $v\geq\varepsilon>0$ in
$Q_-(         {R})$ (otherwise we can deal with $v+\e$ and let $\e\to 0$ at
the end). Let $q>1$ and $\psi\in Y^{s,\g}_{0}(B_R)\cap
L^\infty(B_{  {R}})$ be a nonnegative {radial}
cutoff function such that $\text{supp}\, (\psi)\subseteq
B_{R}$ with $r<R$,          {$0\leq\psi\leq 1$}, $\psi=1$ in
$B_r$ and
\begin{equation}\label{propertyPsi}
\dfrac{(\psi(x)-\psi(y))^2}{|x-y|^2}\le
\frac{C}{(R-r)^2}.
\end{equation}

     Using $\psi^{q+1}v^{-q}$ as a test function in \eqref{peso1}, it
follows that
$$
\int_{B_R}\psi^{q+1}v^{-q}v_t d\mu+
{\frac{a_{N,s}}{2}}\int_{\mathbb{R}^{N}}\int_{\mathbb{R}^{N}}
(v(x,t)-v(y,t))(\psi^{q+1}(x)v^{-q}(x,t)-\psi^{q+1}(y)v^{-q}(y,t))
d\nu\ge 0.
$$
Hence, using a pointwise inequality proved in \cite[Lemma 3.3]{FK}, it can be deduced that
\begin{equation*}\begin{split}
\dyle \frac{1}{q-1}\int_{B_R}&\psi^{q+1}(v^{1-q})_t
d\mu+\frac{1}{q-1}
  {\frac{a_{N,s}}{2}}\int_{B_r}\int_{B_r}
(v^{\frac{1-q}{2}}(x,t)-v^{\frac{1-q}{2}}(y,t))^2d\nu \\
&\dyle \le
C          {q}  {\frac{a_{N,s}}{2}}\int_{\mathbb{R}^{N}}\int_{\mathbb{R}^{N}}\left(
\left(\frac{v(x,t)}{\psi(x)}\right)^{1-q}  {+}\left(\frac{v(y,t)}{\psi(y)}\right)^{1-q}\right)
(\psi(x)-\psi(y))^2 d\nu.
\end{split}\end{equation*}
Furthermore,
\begin{equation*}\begin{split}
\dyle\int_{\mathbb{R}^{N}}\int_{\mathbb{R}^{N}}&\left(
\left(\frac{v(x,t)}{\psi(x)}\right)^{1-q}  {+}\left(\frac{v(y,t)}{\psi(y)}\right)^{1-q}\right)
(\psi(x)-\psi(y))^2 d\nu \\
&\le \dyle          {2}\int_{B_R} \int_{B_R}
\frac{v^{1-q}(x,t)}{|x|^\g}\frac{(\psi(x)-\psi(y))^2dxdy}{|y|^\g|x-y|^{N+2s}}\\
&+         {4}\int_{\ren\backslash B_R}
\int_{B_R}
\frac{v^{1-q}(x,t)}{|x|^\g}\frac{(\psi(x)-\psi(y))^2dxdy}{|y|^\g|x-y|^{N+2s}}.
\end{split}\end{equation*}
We set
$$
I=\int_{B_R} \int_{B_R}
\frac{v^{1-q}(x,t)}{|x|^\g}\frac{(\psi(x)-\psi(y))^2dxdy}{|y|^\g|x-y|^{N+2s}},
$$
and
$$
J=\int_{{\ren\backslash B_R}} \int_{{B_R}}
\frac{v^{1-q}(x,t)}{|x|^\g}\frac{(\psi(x)-\psi(y))^2dxdy}{|y|^\g|x-y|^{N+2s}}.
$$
Let begin by estimating the term $J$. Taking into consideration that $\dfrac{1}{|y|^\g}\le
\dfrac{1}{|x|^\g}$ for $x\in B_R$ and $y\in \ren\backslash B_R$
and using Fubini, we reach that
\begin{equation*}
\begin{array}{lll}
J &\le &\dyle \int_{{B_R}}
\frac{v^{1-q}(x,t)}{|x|^{2\g}}\int_{\{\ren\backslash B_R\}\cap
\{|x-y|>R-r\}} \frac{(\psi(x)-\psi(y))^2dydx}{|x-y|^{N+2s}} \\
\\& + &\dyle \int_{{B_R}}
\frac{v^{1-q}(x,t)}{|x|^{2\g}}\int_{\{\ren\backslash B_R\}\cap
\{|x-y|\le R-r\}}
\frac{(\psi(x)-\psi(y))^2dydx}{|x-y|^{N+2s}}\\
&&\\
&\le & J_1+J_2.
\end{array}
\end{equation*}
Setting $\rho=|x-y|$, we get
\begin{equation*}
\begin{array}{lll}
J_1 & \le & \dyle 4\int_{{B_R}}
\frac{v^{1-q}(x,t)}{|x|^{2\g}}\int_{\{\ren\backslash B_R\}\cap
\{|x-y|>R-r\}} \frac{dydx}{|x-y|^{N+2s}}\\ \\& \le & \dyle
4\int_{{B_R}} \frac{v^{1-q}(x,t)}{|x|^{2\g}}dx
\int_{R-r}^\infty\rho^{-1-2s}d\rho\\ \\ &\le & \dyle
\dfrac{C}{(R-r)^{2s}}\int_{{B_R}}
\frac{v^{1-q}(x,t)}{|x|^{2\g}}dx.
\end{array}
\end{equation*}
We estimate now the term $J_2$. Using \eqref{propertyPsi}, it
follows that
\begin{equation*}
\begin{array}{lll}
J_2 &\le & \dyle \frac{         {C}}{(R-r)^2}\int_{{B_R}}
\frac{v^{1-q}(x,t)}{|x|^{2\g}}\int_{\{\ren\backslash B_R\}\cap
\{|x-y|\le R-r\}} \frac{|x-y|^2dydx}{|x-y|^{N+2s}}\\ \\ &\le &
\dyle \frac{         {C}}{(R-r)^2}\int_{{B_R}}
\frac{v^{1-q}(x,t)}{|x|^{2\g}}dx \int_0^{R-r}\rho^{1-2s}d\rho\\
\\ &\le & \dyle \dfrac{C}{(R-r)^{2s}}\int_{{B_R}}
\frac{v^{1-q}(x,t)}{|x|^{2\g}}dx,
\end{array}
\end{equation*}
         {with $C$ possibly changing from line to line.} Hence
$$
J\le\dfrac{C}{(R-r)^{2s}}\int_{{B_R}}v^{1-q}(x,t)d\mu.
$$
We deal now with $I$. Using the definition of $\psi$, we get easily that
\begin{equation*}
\begin{array}{lll}
I&=& \dyle\iint_{B_R\times B_R\backslash B_r\times B_r}
\frac{v^{1-q}(x,t)}{|x|^\g}\frac{(\psi(x)-\psi(y))^2dxdy}{|y|^\g|x-y|^{N+2s}}\\
\\ &=& \dyle          {\int_{B_R\backslash B_r}\int_{B_r}}
\frac{v^{1-q}(x,t)}{|x|^\g}\frac{(\psi(x)-\psi(y))^2dxdy}{|y|^\g|x-y|^{N+2s}}
+\int_{B_R\backslash B_r}\int_{B_R\backslash
B_r}\frac{v^{1-q}(x,t)}{|x|^\g}\frac{(\psi(x)-\psi(y))^2dxdy}{|y|^\g|x-y|^{N+2s}}\\
\\ & + & \dyle          {\int_{B_r}\int_{B_R\backslash B_r}}
\frac{v^{1-q}(x,t)}{|x|^\g}\frac{(\psi(x)-\psi(y))^2dxdy}{|y|^\g|x-y|^{N+2s}}\\
&&\\
&=& I_1+I_2+I_3.
\end{array}
\end{equation*}
Let begin by estimate $I_1$. Since $(x,y)\in B_r\times
B_R\backslash B_r$, then $|x|\le |y|$, hence
\begin{equation*}
\begin{array}{lll}
I_1 &          {\leq} & \dyle         {\int_{B_R\backslash B_r}\int_{B_r}}
\frac{v^{1-q}(x,t)}{|x|^{2\g}}\frac{(\psi(x)-\psi(y))^2dxdy}{|x-y|^{N+2s}}\\
\\ &\le & \dyle \int_{{B_r}}
\frac{v^{1-q}(x,t)}{|x|^{2\g}}\int_{{\{B_R\backslash B_r\}\cap
\{|x-y|>R-r\}}} \frac{(\psi(x)-\psi(y))^2dydx}{|x-y|^{N+2s}}\\ \\&
+ &\dyle \int_{{B_r}}
\frac{v^{1-q}(x,t)}{|x|^{2\g}}\int_{{\{B_R\backslash B_r\}\cap
\{|x-y|\le R-r\}}}
\frac{(\psi(x)-\psi(y))^2dydx}{|x-y|^{N+2s}}\\
&&\\
&\le & I_{11}+I_{12}.
\end{array}
\end{equation*}

     As in the previous computations, by setting $\rho=|x-y|$ and using
the fact that $\psi$ is bounded, we conclude that
$$
I_{11}\le 4\int_{{B_r}}
\frac{v^{1-q}(x,t)}{|x|^{2\g}}\int_{R-r}^\infty
\rho^{-1-2s}d\rho\le \dfrac{C}{(R-r)^{2s}}\int_{{B_r}}
\frac{v^{1-q}(x,t)}{|x|^{2\g}}dx.
$$
In the same way and using the fact that
$\displaystyle\dfrac{(\psi(x)-\psi(y))^2}{|x-y|^2}\le
\frac{C}{(R-r)^2}$, we get
$$
I_{12}\le \dfrac{C}{(R-r)^2}\int_{{B_r}}
\frac{v^{1-q}(x,t)}{|x|^{2\g}}\int_0^{R-r}\rho^{1-2s}d\rho\le
\dfrac{C}{(R-r)^{2s}}\int_{{B_r}}
\frac{v^{1-q}(x,t)}{|x|^{2\g}}dx.
$$
Therefore,
$$
I_{1}\le \dfrac{C}{(R-r)^{2s}}\int_{{B_R}}
\frac{v^{1-q}(x,t)}{|x|^{2\g}}dx.
$$
We deal now with $I_2$. Since $\frac 12\le r\le |x|, |y|\le R<1$, then
\begin{equation*}
\begin{array}{lll}
I_2 & \le & 2^\gamma \dyle\int_{B_R\backslash
B_r}\int_{B_R\backslash B_r}
\frac{v^{1-q}(x,t)}{|x|^{2\g}}\frac{(\psi(x)-\psi(y))^2dxdy}{|x-y|^{N+2s}}\\
\\ &\le & \dyle \int_{{B_R\backslash B_r}}
\frac{v^{1-q}(x,t)}{|x|^{2\g}}\int_{{\{B_R\backslash B_r\}\cap
\{|x-y|>R-r\}}} \frac{(\psi(x)-\psi(y))^2dydx}{|x-y|^{N+2s}}\\ \\&
+ &\dyle \int_{{B_R\backslash B_r}}
\frac{v^{1-q}(x,t)}{|x|^{2\g}}\int_{{\{B_R\backslash B_r\}\cap
\{|x-y|\le R-r\}}}
\frac{(\psi(x)-\psi(y))^2dydx}{|x-y|^{N+2s}}\\
&&\\
&\le & \dyle          {\dfrac{C}{(R-r)^{2s}}\int_{{B_R}}
\frac{v^{1-q}(x,t)}{|x|^{2\g}}dx,}
\end{array}
\end{equation*}
         {by repeating the computations performed for $I_1$.} Let consider now the term $I_3$ which is the most complicated.
\begin{equation*}
\begin{array}{lll}
I_3 &=& \dyle \int_{{r\le |x|\le
R}}\frac{v^{1-q}(x,t)}{|x|^{\g}}\Big( \int_{|y|\le \frac{|x|}{2}}
\dfrac{(\psi(x)-\psi(y))^{2}}{|x-y|^{N+2s}|y|^{\gamma}}dy\Big)
dx\\ \\ &+& \dyle \int_{{r\le |x|\le
R}}\frac{v^{1-q}(x,t)}{|x|^{\g}}\Big( \int_{\frac{|x|}{2}\le
|y|\le r}
\dfrac{(\psi(x)-\psi(y))^{2}}{|x-y|^{N+2s}|y|^{\gamma}}dy\Big)
dx\\
&&\\
&=& I_{31}+I_{32}.
\end{array}
\end{equation*}
If $|y|\le \dfrac{|x|}{2}$, then $|x-y|\ge \dfrac{|x|}{2}\ge
\dfrac{r}{2}\ge \dfrac{1}{4}$, and thus,
\begin{equation*}
\begin{array}{lll}
I_{31} &\le & C \dyle\int_{{r\le |x|\le
R}}\frac{v^{1-q}(x,t)}{|x|^{\g}}\Big( \int_{|y|\le
\frac{|x|}{2}}\dfrac{1}{|y|^{\gamma}}dy\Big) dx\\ \\ &\le & C
\dyle \int_{{r\le |x|\le R}}\frac{v^{1-q}(x,t)}{|x|^{\g}}\Big(
\int_0^{\frac{|x|}{2}}\rho^{N-1-\g}d\rho\Big)dx\\ \\ &\le &
\dyle\dfrac{C}{(R-r)^{2s}}\int_{{B_R}}
\frac{v^{1-q}(x,t)}{|x|^{2\g}}dx. \end{array}\end{equation*}
To
estimate $I_{32}$, we use the fact that $\dfrac{1}{|y|^\gamma}\le
\dfrac{2^\gamma}{|x|^\gamma}$, hence
\begin{equation*}
\begin{array}{lll}
I_{32} &\le & C \dyle \int_{{r\le |x|\le
R}}\frac{v^{1-q}(x,t)}{|x|^{2\g}} \Big( \int_{\frac{|x|}{2}\le
|y|\le r}
         {\dfrac{(\psi(x)-\psi(y))^{2}}{|x-y|^{N+2s}}dy}\Big)
dx\\ \\ &\le & C \dyle \int_{{r\le |x|\le
R}}\frac{v^{1-q}(x,t)}{|x|^{2\g}} \Big( \int_{\{\frac{|x|}{2}\le
|y|\le r\}\cap \{|x-y|>R-r\}}
         {\dfrac{(\psi(x)-\psi(y))^{2}}{|x-y|^{N+2s}}dy}\Big)
dx\\ \\ &+& C\dyle \int_{{r\le |x|\le
R}}\frac{v^{1-q}(x,t)}{|x|^{2\g}} \Big( \int_{\{\frac{|x|}{2}\le
|y|\le r\}\cap \{|x-y|\le R-r\}}
         {\dfrac{(\psi(x)-\psi(y))^{2}}{|x-y|^{N+2s}}dy}\Big)
dx.
\end{array}
\end{equation*}
Using the same computations
as in the estimates of $I_{11}$ and $I_{12}$ it follows that
$$
I_{32}\le \dfrac{C}{(R-r)^{2s}}\int_{{B_R}}
\frac{v^{1-q}(x,t)}{|x|^{2\g}}dx.
$$
Combining the estimates above, there results that
\begin{eqnarray*}
&\dyle
\int_{B_R}\psi^{q+1}(v^{1-q})_td\mu+\frac{a_{N,s}}{2}\int_{B_r}\int_{B_r}
(v^{\frac{1-q}{2}}(x,t)-v^{\frac{1-q}{2}}(y,t))^2d\nu\\
&\le
\dyle\frac{C q^2}{(R-r)^{2s}}  {\frac{a_{N,s}}{2}}\int_{B_R}
v^{1-q}(x,t)d\mu.
\end{eqnarray*}

     Set now $\theta(t):=[\min\{\dfrac{t+R^{2s}}{R^{2s}-r^{2s}}, 1\}]_+$. Then multiplying the last inequality by $\theta$, integrating in
time in $(-R^{2s}, t)$ with $t\in (-r^{2s},0)$, and noticing that $\theta(t)=1$ for $t\ge -r^{2s}$ and
$|\theta'(t)|\le \frac{1}{R^{2s}-r^{2s}}$, it follows that
\begin{equation}\label{supw}\begin{split}
\sup_{t\in I_-(r)}\int_{B_r}(v^{1-q})d\mu&+
\frac{a_{N,s}}{2}\int\!\!\!\int_{Q_-(r)}\int_{B_r}
(v^{\frac{1-q}{2}}(x,t)-v^{\frac{1-q}{2}}(y,t))^2d\nu d t\\
&\le
C q^2  {\frac{a_{N,s}}{2}}\left(\frac{1}{(R-r)^{2s}}+\frac{1}{R^{2s}-r^{2s}}\right)         {\iint_{Q_-(R)}}
v^{1-q}(x,t)d\mu         {dt}.
\end{split}\end{equation}
Recalling that $\tau:=1+\frac{2s}{N}$, let us define
$w:=v^{\frac{1-q}{2}}$. Then, applying H\"{o}lder inequality, we get
\begin{eqnarray*}
\displaystyle \int\!\!\!\int_{Q_-(r)}w^{2\tau}d\mu dt&=&
\int\!\!\int_{Q_-(r)}w^2w^{\frac{4s}{N}}d\mu dt\\
&\le& \dyle
\int_{I_-(r)}\Big(\int_{B_r}w^{2}d\mu\Big)^{\frac{2s}{N}}
\Big(\int_{B_r}w^{2^*_s}d\mu\Big)^{\frac{2}{2^*_s}}dt.
\end{eqnarray*}
Since $\g>0$ and $R\le 1$, we conclude that
\begin{eqnarray*}
\dyle \int\!\!\!\int_{Q_-(r)}w^{2\tau}d\mu dt&\le & \int_{I_-(r)}\Big(\int_{B_r}w^2(x,t)d\mu\Big)^{\frac{2s}{N}} \Big(\dyle
\int_{B_r}\frac{w^{2^*_s}}{|x|^{2^*_s\g}}dx\Big)^{\frac{2}{2^*_s}}dt.
\end{eqnarray*}
Now, using the Sobolev inequality obtained in Theorem
\ref{sobolev} in the Appendix,
\begin{eqnarray*}
\dyle\int\!\!\!\int_{Q_-(r)}w^{2\tau}d\mu dt&\le& C\sup_{t\in
I_-(r)}\dyle\Big(\int_{B_r}w^2(x,t)d\mu\Big)^{\frac{2s}{N}}\\
&&\times \dyle \Big( \int\!\!\!\int_{Q_-(r)}         {\int_{B_r}}(w(x,t)-w(y,t))^2         {d\nu}
dt+ r^{-2s}\int \!\!\!\int_{Q_-(r)}w^2d\mu dt\Big).
\end{eqnarray*}
Applying \eqref{supw} twice at this inequality, and recalling that $\dfrac{1}{2}\leq r\leq 1$, it can be checked that
\begin{equation*}
\dyle\int\!\!\!\int_{Q_-(r)}w^{2\tau}d\mu dt\le A(q,r,R,N,s)\Big(
\int\!\!\!\int_{Q_-(R)}w^2d\mu dt\Big)^\tau
\end{equation*}
where
$$
A(q,r,R,N,s)=C(N,s,\gamma) q^2\left(\frac{1}{(R-r)^{2s}}+\frac{1}{R^{2s}-r^{2s}}\right)^\tau.
$$
Setting $p:=q-1$, we conclude the proof.
\end{proof}

As an application of the previous estimate, we
reach a control of $\sup_{Q_-(r)} v^{-1}$. More precisely, we have the following result.
\begin{Lemma}\label{lema2}
Assume that $\frac 12\le r<R\le 1$ and $p\in (0,1]$. Then, there
exists a constant $C=C(N,s,\gamma)>0$ such that every $v\gneq 0$
supersolution to the problem \eqref{peso1} satisfies
\begin{equation}\label{estim2}
\sup_{Q_-(r)} v^{-1}\le \left(\frac{C}{\alpha(r,R)}\right)^{\frac
1p}\left(\iint_{Q_-({R})}v^{-p}d\mu{dt}\right)^{\frac{1}{p}},
\end{equation}
where
$$
\alpha(r,R)= \left\{
\begin{array}{lll}
(R-r)^{N+2s} &\mbox{  if   } & s\ge \frac 12,\\
(R^{2s}-r^{2s})^{\frac{N+2s}{2s}} &\mbox{  if   } & s<\frac 12.\\
\end{array}
\right.
$$
\end{Lemma}
\begin{proof}
We consider
$$
\mathcal{M}(r,p):=
\left(\int\!\!\!\int_{Q_-(r)}v^{-p}d\mu\,  {dt}\right)^{\frac{1}{p}}.
$$
By Lemma \ref{negative}, we have
$$
\mathcal{M}(r,\tau p)\le A^{\frac{1}{p}}\mathcal{M}(R,p),
$$
where $\tau:=1+\frac{2s}{N}$. Construct now the sequences $\{r_i\}_{i\in \mathbb{N}}$ and
$\{p_i\}_{i\in \mathbb{N}}$ by setting $r_0:=R>r_1>r_2>...>r$ and
$p_i:=p\tau^i$. Using again Lemma \ref{negative}, we obtain
$$
\mathcal{M}(r,p_{m+1})\le\mathcal{M}(r_{m+1},p_{m+1})\le
A_{{m}}^{\frac{1}{{\tau^m p}}}\mathcal{M}(r_m,p_m),
$$
{with
$A_m:=C(p_m+1)^2\left((r_m-r_{m+1})^{-2s}+(r_m^{2s}-r_{m+1}^{2s})^{-1}\right)$.}
Iterating this and following the arguments in \cite[Theorem
3.5]{FK}, we conclude the result.
\end{proof}

We prove now a control of a small positive exponents of $u$.
\begin{Lemma}\label{tres}
Suppose that $\frac 12\le r<R\le   {1}$, and fix $q\in
(0,\tau^{-1}{]}$, {with $\tau:=1+\frac{2s}{N}$}. Then, if $v\gneq
0$ is a supersolution to \eqref{peso1}, we have
\begin{equation}\label{est3}
\dyle\Big(\int\!\!\!\int_{Q_+(r)}v^{q\tau}d\mu dt\Big)^{\frac{1}{\tau}}\le
\alpha \int\!\!\!\int_{Q_+(R)}v^{q}d\mu dt,
\end{equation}
where
$$
\alpha=\alpha(N,s,r,R,\gamma)=C(  {N,s,\gamma})\left(\frac{1}{(R-r)^{2s}}+\frac{1}{R^{2s}-r^{2s}}\right).
$$
\end{Lemma}
\begin{proof}
The proof follows similarly to the one of Lemma \ref{negative}
(see \cite[Proposition 3.6]{FK} for a detailed proof with the
Lebesgue measure). We set $a:=(1-q)\in [1-\tau^{-1}, 1)$ and
$w(x,t):=v^{\frac{{1-a}}{2}}$. Then, using $v^{{-a}}\psi^{{2}}$
{(with $\psi$ defined in the proof of Lemma \ref{negative})} as a
test function in \eqref{peso1}, we reach that
\begin{eqnarray*}
&\dyle\sup_{t\in I_+(r)}\int_{B_r}w^2(x,t)d\mu+
C(q)  {\frac{a_{N,s}}{2}}\int_{Q_+(r)}\int_{B_r}
(w^2(x,t)-w^2(y,t))^2d\nu d{t}\\
&\le\dyle
C(q)  {\frac{a_{N,s}}{2}}\left(\frac{1}{(R-r)^{2s}}+\frac{1}{R^{2s}-r^{2s}}\right)\int_{Q_+(R)}
w^2(x,t)d\mu  {dt}.
\end{eqnarray*}
Using  Theorem \ref{sobolev} and proceeding as in Lemma
\ref{negative}, we get \eqref{est3}.
\end{proof}

Define
$$
\mathcal{H}(r,  {q})=\Big(\int_{Q_+(r)}v^{q}d\mu
dt\Big)^{\frac{1}{q}}.
$$

     From \eqref{est3}, we get $\mathcal{H}(r,\tau
  {q})\le
\alpha^{\frac{1}{  {q}}}\mathcal{H}(R,  {q})$. Let define $q_j:=\tau^{-j}$ and
$$         {r_j:=\begin{cases}
r+\frac{R-r}{2^j}\mbox{ if }s\geq\frac 12,\\
\left(r^{2s}+\frac{R^{2s}-r^{2s}}{2^j}\right)^{1/2s}\mbox{ if }s<\frac 12.
\end{cases}}$$
By Lemma \ref{tres} for $r_n$ and $r_{n-1}$, it follows
that
\begin{equation}\label{iter200}
\mathcal{H}(r_n,q_1\tau)\le
\alpha_n^{\tau}\mathcal{H}(r_{n-1},{q_1}),
\end{equation}
where $\alpha_n=C(N,s,\gamma)\left(\dfrac{1}{(r_{n-1}-r_n)^{2s}}+\dfrac{1}{r_{n-1}^{2s}-r_n^{2s}}\right).$
By using the definition of $r_n$ and considering that $r\ge \frac
12$, we get
$$
\alpha_n\le C(N,s,\gamma)\dfrac{2^{2ns}}{(R-r)^{2s}}.
$$
Hence
\begin{equation}\label{itt}
\mathcal{H}(r_n,q_1\tau)\le
\left(C(N,s,\gamma)\dfrac{2^{2ns}}{(R-r)^{2s}}\right)^{\tau}
\mathcal{H}(r_{n-1},{q_1}).
\end{equation}
Iterating this inequality (see \cite[Theorem 3.7]{FK}), we reach the next result.
\begin{Lemma}\label{lema5}
Assume that $\frac{1}{2}\le r<R\le 1$ and $q\in(0,\tau^{-1})$, {with
$\tau:=1+ \frac{2s}{N}$}. {Then, every supersolution $v\gneq 0$ of problem \eqref{peso1} satisfies}
\begin{equation}\label{estim5}
\int\!\!\!\int_{Q_+(r)} v{\,d\mu dt}\le \left(\frac{C}{|Q_+(1)|_{d\mu\times
dt}\overline{\alpha}(r,R)}\right)^{\frac{1-  {q}}{
{q}}}\left(\int\!\!\!\int_{Q_+(R)}v^{q}d\mu dt\right)^{\frac{1}{q}},
\end{equation}
where $C=C(N,s,\gamma)>0$ and
$$
{\overline{\alpha}(r,R)=
\begin{cases}
(R-r)^{\omega_1}\qquad\hbox{if }s\geq\frac{1}{2},\\
(R^{2s}-r^{2s})^{\omega_2}\qquad\hbox{if }s<\frac{1}{2},
\end{cases}}
$$
with $\omega_1, \omega_2>0$ depending only on $s,N$.
\end{Lemma}

In order to apply Lemma \ref{clasi} we need to estimate
\begin{equation}\label{estQlog}
|Q_+(1)\cap\{\log v<-m-a\}|_{d\mu\times dt} \;\;\mbox{  and
}\;\;|Q_-(1)\cap\{\log v>m-a\}|_{d\mu\times dt},
\end{equation}
where $m,{a}>0$ {and $v$ is a supersolution to \eqref{peso1}}. To
do this, we need the following auxiliary result, whose proof immediately follows from Lemma 4.1 in \cite{FK}.
\begin{Lemma}\label{estLog}
Let $I\subset\mathbb{R}$ and $\psi:\mathbb{R}^N\rightarrow [0,+\infty)$ be a continuous function satisfying supp$(\psi)= \overline{B_R}$ for some $R>0$
and $|||\psi|||_{Y_0^{s,\gamma}(\mathbb{R}^N)}\leq C$. Then, for $v:\mathbb{R}^N\times I\rightarrow[0,+\infty)$, the following inequality holds,
\begin{equation*}\begin{split}
\int_{\mathbb{R}^N}\int_{\mathbb{R}^N}&(v(x,t)-v(y,t))(-\psi^2(x)v^{-1}(x,t)+\psi^2(y)v^{-1}(y,t))\,d\nu\\
&\geq \int_{B_R}\int_{B_R}\psi(x)\psi(y)\left(\log{\frac{v(y,t)}{\psi(y)}}-\log{\frac{v(x,t)}{\psi(x)}}\right)^2\,d\nu-3\int_{\mathbb{R}^N}\int_{\mathbb{R}^N}(\psi(x)-\psi(y))^2\,d\nu.
\end{split}\end{equation*}
\end{Lemma}

With this result, we can establish the estimates in \eqref{estQlog} {(see  \cite[Proposition 4.2]{FK} for more details)}.

\begin{Lemma}\label{log}
Assume that $v$ is a supersolution to \eqref{peso1} in the
cylinder $Q:= B_2\times (-1,1)$, then there exists a positive
constant $C=C(N,s,         {\gamma})$ such that for some constant $a=a(v)$, we have
\begin{equation}\label{log1}
\forall m>0: |Q_+(1)\cap\{\log v<-m-a\}|_{d\mu\times dt}\le
\frac{  {C}|B_1|_{d\mu}}{m},
\end{equation}
and
\begin{equation}\label{log2}
\forall m>0: |Q_-(1)\cap\{\log v>m-a\}|_{d\mu\times dt}\le
\frac{  {C}|B_1|_{d\mu}}{m}.
\end{equation}
\end{Lemma}
\begin{proof}
{Suppose that $v\geq\varepsilon>0$ in $Q$}. Let $\psi$ be such that
$\psi^2=((\frac 32 -         {\frac{|x|}{2}})\wedge 1)\vee 0$,
  and let us denote
$w(x,t):=-\log\dfrac{v(x,t)}{\psi(x)}$. Using
$\dfrac{\psi^2}{{v}}$ as a test function in \eqref{peso1} and
{noticing that $\text{supp} (\psi^2)\subseteq B_{3/2}$}, {by Lemma
\ref{estLog} and the fact that
$|||\psi|||_{Y_0^{s,\gamma}(\mathbb{R}^N)}\leq C$, there follows}
\begin{eqnarray*}
&\dyle\int_{B_{3/2}}\psi^{2}w_t
d\mu+  {\frac{a_{N,s}}{2}}\int_{B_{
3/2}}\int_{B_{3/2}}\psi(x)\psi(y)
(w(x,t)-w(y,t))^2d\nu\\
&\displaystyle \le 3{\frac{a_{N,s}}{2}}\dyle
\int_{\mathbb{R}^{N}}\int_{\mathbb{R}^{N}}(\psi(x)-\psi(y))^2 d\nu
\leq C.
\end{eqnarray*}
%Since $\psi\le 1$, then
%\begin{eqnarray*}
%&\dyle \int_{B_{\frac 32}}\psi^{2}w_t
%d\mu+  {\frac{a_{N,s}}{2}}\int_{B_{\frac
%32}}\int_{B_{\frac 32}}\psi(x)\psi(y)
%(w(x,t)-w(y,t))^2 d\nu\\
%&\le C|B_1|_{d\mu}.
%\end{eqnarray*}
We set $W(t):=\dfrac{\int_{B_{
3/2}}\psi^{2}w(x,t)d\mu}{\int_{B_{3/2}}\psi^{2}d\mu}$, then by
the Poincar\'e type inequality obtained in Theorem \ref{PW}, we
reach that
\begin{eqnarray*}
\dyle\int_{B_{3/2}}\psi^{2}w_t d\mu+C\int_{B_{
3/2}}(w(x,t)-W(t))^2\psi(x)^{  {2}}d\mu\le
C.
\end{eqnarray*}
Let $(t_1,t_2)\subset (-1,1)$. Integrating in time the previous
inequality, {dividing by $\displaystyle
\int_{B_{3/2}}\psi^2\,d\mu$, and noticing that
$$
\int_{B_{3/2}}\psi^2\,d\mu\leq 2^{N-2\gamma}|B_1|_{d\mu},
$$
one gets}
$$
\dfrac{W(t_2)-W(t_1)}{t_2-t_1}+\frac{C_1}{|B_1|_{d\mu}(t_2-t_1)}\int_{t_1}^{t_2}\int_{B_1}(w(x,t)-W(t))^2d\mu\le
C_2.
$$
We can suppose that $W$ is differentiable. In the contrary case, it is possible to follow a discretization  argument as in
\cite[Proposition 4.2]{FK}. By letting $t_2\to t_1$, we get
\begin{equation}\label{derW0}
W'(t)+\frac{C_1}{|B_1|_{d\mu}}\int_{B_1}(w(x,t)-W(t))^2d\mu\le
C_2 \:\qquad a.e \mbox{  in  }(-1,1).
\end{equation}
Define $\tilde{W}(t)=V(t)-C_2t$ and $\tilde{w}(x,t)=v(x,t)-C_2t$. From \eqref{derW0}, it follows that
\begin{equation}\label{derW}
\tilde{W}'(t)+\frac{C_1}{|B_1|_{d\mu}}\int_{B_1}(\tilde{w}(x,t)-\tilde{W}(t))^2d\mu\leq 0 \:\qquad a.e \mbox{  in  }(-1,1).
\end{equation}
Notice that from \eqref{derW} we
deduce $\tilde{W}'(t)\leq 0$, and therefore calling $a(v):=W(0)$
there results
$$  {\tilde{W}(t)\leq W(0)=:a({v})\mbox{  for all }t\in (0,1).}$$
         {Let $t\in (0,1)$.} Then if we define
$$
G^+_{  {m}}(t):=\{x\in B_1(0): \tilde{w}(x,t)>m+a\},
$$
for $x\in   {G}^+_m(t)$, we have
$$
\tilde{w}(x,t)-\tilde{W}(t)\ge m+a-\tilde{W}(t)>0.
$$
Thus
$$
\tilde{W}'(t)+\frac{C_1}{|B_1|_{d\mu}}|G^+_m(t)|_{d\mu}(m+a-\tilde{W}(t))^2\le
0.
$$
Hence
$$
\frac{-\tilde{W}'(t)}{(m+a-\tilde{W}(t))^2}\ge
\frac{C_1}{|B_1|_{d\mu}}|  {G}^+_m(t)|_{d\mu}.
$$
Integrating the previous differential inequality {for $t\in
(0,1)$} and substituting $\tilde{w}$ by its value, yields
$$
|Q_+(1)\cap\{\log
v+C_2t<-  {m}-a\}|_{  {d\mu\times
dt}}\le\frac{C_1|B_1|_{d\mu}}{m}.
$$
Now
\begin{eqnarray*}
|Q_+(1)\cap\{\log
v<-  {m}-a\}|_{  {d\mu\times dt}} &\le &
|Q_+(1)\cap\{\log
v+C_2t<-           {\frac{m}{2}}-a\}|_{  {d\mu\times dt}}\\
&+&|Q_+(1)\cap\{C_2t>\frac{m}{2}\} |_{  {d\mu\times dt}}\\ &\le
&\frac{C|B_1|_{d\mu}}{m},
\end{eqnarray*}
what finishes the proof of \eqref{log1}. Estimate \eqref{log2}
follows using the same approach.
\end{proof}

We are now able to prove the weighted weak Harnack inequality.

\begin{proof}[Proof of Theorem \ref{harnack}]
Roughly speaking, the key to prove this result will be to define appropriate functions and parameters so that we can deduce the result from Lemma \ref{clasi}.
Indeed, we divide the proof in two cases.   {Let $0<r<1$ such that $B_r\subset\Omega$.}
\begin{enumerate}
\item Assume first that $s\ge \frac 12$.

     We set $\theta_1=\theta_2=\frac 12$ and define $U_1(r)=B_r\times
(1-r^{2s}, 1)$,  $U_2(r)=B_r\times (-1, -1+r^{2s})$. In the same
way we consider $U_{  {1}}(1)=Q_+(1)$ and ${U}_2(1)=Q_-(1)$.

     Let $w_1:=e^{-a}v^{-1}, w_2:=e^{a}v$ where $a=a(v)$ was defined in
Lemma \ref{log}. From this result we obtain that
$$
|Q_+(1)\cap\{\log w_1>m\}|_{d\mu\times dt}\le
\frac{C|B_1|_{d\mu}}{m},
$$
and
$$
|Q_-(1)\cap\{\log w_2>m\}|_{d\mu\times dt}\le
\frac{C|B_1|_{d\mu}}{m}.
$$
Using Lemma \ref{lema2}, it follows that $(w_1,U_1(r))$ satisfies
the conditions of Lemma \ref{clasi} with $p_0=\infty$ and $\eta$
any positive constant. Moreover, by Lemma \ref{lema5}, $(w_2,
{U}_2(r))$ satisfies the same conditions with $  {p_0}=1$ and
$\eta=\frac{N}{N+{2s}}<1$. Hence we conclude that
$$
\sup_{U_1(\frac{1}{2})} w_1\le C\;\;\mbox{  and
}\;\;\|w_2\|_{L^1({U_2}(\frac{1}{2}), d\mu)}\le  {\tilde{C}}.
$$
Using these estimates and the definitions of $w_1$ and $w_2$ , we
get
$$
\|v\|_{L^1({U}_2(\frac 12),  {d\mu})}\le C
\inf_{U_{  {1}}(\frac 12)} v,
$$
and the result follows in this case.
\item If $0<s<\frac 12$, we have to change the domains by  setting
$\theta_1=\theta_2=(\frac 12)^{2s}$ and
$U_1(r)=B_{r^{\frac{1}{2s}}}\times (1-r, 1)$,
${U}_2(r)=B_{r^{\frac{1}{2s}}}\times (-1,-1+r)$. Then
the same arguments as in the previous case allow us to conclude.
\end{enumerate}
\end{proof}
From the weighted weak Harnack inequality, we immediately deduce the next Corollary.
         {\begin{Corollary}\label{coroHar}
Let $\l\le
\Lambda_{N,s}$. Assume that $u$ is a nonnegative function such
that $u\not\equiv 0$,\, $u\in L^1(\Omega\times (0,T))$ and
$\dfrac{u}{|x|^{2s}}\in L^1(\Omega\times (0,T))$. If $u$ satisfies $u_t+(-\Delta)^{s}
u-\l\dfrac{u}{|x|^{2s}}\geq 0$ in the weak sense in $\Omega\times (0,T)$, then there exists $r_1>0$ and $t_2>t_1>0$, and a constant  $C=C(N,r_1,t_1,t_2)$ such that for each cylinder $B_{r}(0)\times (t_1,t_2)\subset\subset \Omega\times (0,T)$, $0<r<r_1$,
$$u\geq C |x|^{-\frac{N-2s}{2}+\alpha}\hbox{ in }  B_{r}(0)\times (t_1,t_2),$$ where $\alpha$ is the singularity of the homogeneous problem given in Lemma \ref{singularity}.
In particular, for $r$ conveniently small we can assume that $u>1$ in $B_{r}(0)\times (t_1,t_2).$
\end{Corollary}}
Finally, to end this section, we can establish a boundedness condition on the solutions of \eqref{peso1}.
\begin{Proposition}\label{boundednessV}
Let $v\in \mathcal{C}(  {[0,T)};L^2(\ren,d\mu))\cap L^2(0,T;Y^{s,\g}_{0}(\ren))$ be  a
solution to \eqref{peso1}  with $u_0\in L^\infty(\Omega)$.
If $f\in L^r(0,T; L^{q}(\Omega)) \mbox{ with } r, q>1$ and $\dfrac{1}{r}+\dfrac{N}{2qs}< 1$, then $v\in L^\infty(\Omega\times (0,T))$.
\end{Proposition}
\begin{proof}
The proof follows  the same idea of the classical result by D.G. Aronson and J.Serrin in \cite{AS}.
We test with $G_k(v)\in Y_{0}^{s,\gamma}(\Omega)$ in \eqref{peso1}, and defining
$$|||G_k(v)|||^2:=\|G_k(v)\|_{L^\infty(0,T; L^2(\Omega,d\mu))}^2+\|G_k(v)\|^2_{L^2(0,T;Y_0^{s,\g}(\Omega))},$$
the result is obtained as a simplified version of \cite[Theorem 29]{LPPS}.
The presence of the singular term is handled in a straightforward way, so we skip the details.
\end{proof}
\begin{remark}
Apart from the integral version we proved, if the solution is bounded we can prove the strong Harnack inequality by classical arguments. We skip the details because we do not use this inequality in the applications.
\end{remark}
\begin{Corollary}
If $u$ is a solution to \eqref{eq:prob} with a sufficient regular datum $f$, then $u\leq C |x|^{-\frac{N-2s}{2}+\alpha}$ in $\Omega\times (0,T)$.
\end{Corollary}

\section{The Linear Problem: Dependence on the spectral parameter $\lambda$.}\label{4}
     Along this section we will study the problem
\begin{equation}\label{problemBG}
\left\{
\begin{array}{rcl}
 u_t+(-\Delta)^{s} u&=&\l\dfrac{\,u}{|x|^{2s}}+g(x,t)\mbox{ in } \Omega\times (0,T),\\
u(x,t)&=&0\inn (\ren\setminus\Omega)\times[ 0,T),\\
u(x,0)&=&u_0(x)\gneq 0 \mbox{ if }x\in\O,
\end{array}
\right.
\end{equation}
where $g(x,t)$ is a nonnegative function. The goal will be to establish some necessary and sufficient conditions on $g$ and $u_0$ in order to find solutions of this problem. These results correspond to the ones obtained by P. Baras and J. A. Goldstein  for the heat equation in  presence of the inverse square potential (see \cite{BaGo}).

     First, we deal with the necessary summability conditions on $g$ and $u_0$.
\begin{Theorem}\label{necesario}
Let $0<\l\le \Lambda_{N,s}$. Assume that $\tilde{u}$  is a positive weak supersolution to the problem \eqref{problemBG}. Then $g$ and $u_0$
must satisfy
$$\dint_{t_1}^{t_2}\dint_{B_{r}(0)}|x|^{-\gamma}g\,dx\,dt<+\iy,\quad \int_{B_r(0)}|x|^{-\gamma}u_0\,dx<+\iy,$$
for any cylinder $B_{r}(0)\times (t_1,t_2)\subset\subset\Omega\times (0,T)$, where $\gamma$ was defined in \eqref{g1}.
\end{Theorem}
\begin{proof} Fix $\varepsilon>0$. Let consider $\varphi_{n}$, the positive solution to
\begin{equation}\label{eq:super}
\left\{\begin{array}{rcl}
-(\varphi_n)_t+(-\Delta)^s \varphi_n&=&\lambda\dfrac{\varphi_{n-1}}{|x|^{2s}+\frac{1}{n}}+1 \quad\mbox{in }\Omega\times (-\varepsilon,T),\\
\varphi_n&=&0\inn(\ren\setminus\Omega)\times  {(} -\varepsilon,T  {]},\\
\varphi_n(x,T)&=&0\inn\Omega,
\end{array}\right.
\end{equation}
with
\begin{equation}\label{eq:super00}
\left\{\begin{array}{rcl}
-(\varphi_0)_t+(-\Delta)^s \varphi_0&=&1 \quad\mbox{in }\Omega\times (-\varepsilon,T),\\
\varphi_0&=&0\inn(\ren\setminus\Omega)\times  {(} -\varepsilon,T  {]},\\
\varphi_0(x,T)&=&0\inn\Omega.
\end{array}\right.
\end{equation}
Notice that $\varphi_0$ is  a strong solution in $\Omega\times   {(} -\varepsilon,T  {]}$, and therefore every $\varphi_n$ is a strong solution
too (see Appendix A). Furthermore,
 $\varphi_{n-1}\leq \varphi_{n}\leq\varphi$, where  $\varphi$ is the positive solution to
$$\left\{\begin{array}{rcl}
-\varphi_{t}+(-\Delta)^{s}\varphi&-&\lambda\dfrac{\varphi}{|x|^{2s}}=1\inn\Omega\times (-\varepsilon,T),\\
\varphi&=&0\inn(\ren\setminus\Omega)\times  ( -\varepsilon,T  {]},\\
\varphi(x,T)&=&0,\quad \Omega.
\end{array}
\right.
$$
As a consequence of          {Corollary \ref{coroHar}}, for any cylinder $C_{r_1,t_1,t_2}:=B_{r}(0)\times (t_{1},t_2)\subset\subset \Omega\times   {(-\varepsilon,T)}$, $0<r<r_1$, there exists a constant  $A=A(N,s, C_{r_1,t_1,t_2})$,  such that
\begin{equation}\label{singVarphi}
\varphi(x,t)\geq \frac{A}{|x|^\gamma},\quad (x,t)\in C_{r_1,t_1,t_2}, \quad \gamma=\frac{N-2s}{2}-\alpha.
\end{equation}
Since $\varphi_n$ is regular and bounded we can use it as a test function in \eqref{problemBG}, hence  we get
\begin{equation*}\begin{split}
\int_{0}^{T}{\int_\Omega{g\varphi_n\,dx}\,dt}&+\int_\Omega{u_0\varphi_n(x,0)\,dx} \\
&         {\leq} -\int_{0}^{T}{\int_\Omega{\tilde{u}(\varphi_n)_t\,dx}\,dt}+\int_{0}^{T}{\int_{         {\R^N}}{\tilde{u}(-\Delta)^s\varphi_n\,dx}\,dt}
-\lambda \int_{0}^{T}{\int_\Omega{\frac{\tilde{u}\varphi_{n}}{|x|^{2s}}\,dx}\,dt}\\
&\leq-\int_{0}^{T}{\int_\Omega{\tilde{u}(\varphi_n)_t\,dx}\,dt}+\int_{0}^{T}{\int_\Omega{\tilde{u}(-\Delta)^s\varphi_n\,dx}\,dt}-\lambda \int_{0}^{T}{\int_\Omega{\frac{\tilde{u}\varphi_{n-1}}{|x|^{2s}+\frac{1}{n}}\,dx}\,dt}\\
&=\int_{0}^{T}{\int_\Omega{\tilde{u}\,dx}\,dt}=C<+\infty.
\end{split}\end{equation*}
Since both integrals in the left hand side are positive, in particular each one is uniformly bounded.  Thus, $\{g\varphi_{n}\}$ is an increasing sequence uniformly bounded in $L^1(\Omega\times(0,T))$, and applying the Monotone Convergence Theorem and
\eqref{singVarphi} we get
\begin{equation*}\begin{split}
C\dint_{t_1}^{t_2}\dint_{B_{r}(0)}&|x|^{-\frac{N-2s}{2}+\alpha} g\,dx dt\leq\int_{t_1}^{t_2}{\int_{B_r(0)}{g\varphi\,dx}\,dt}\\
&\le
\int_{0}^{T}{\int_\Omega{g\varphi\,dx}\,dt}=\lim_{n\rightarrow\infty}\int_{0}^{T}{\int_\Omega{g\varphi_n\,dx}\,dt}<+\infty.
\end{split}\end{equation*}
Likewise, $\{u_0\varphi_n(x,0)\}$ is also an increasing sequence, uniformly bounded in $L^1(\Omega)$, and thus, choosing $t_1$
and $t_2$ so that $0\in (t_1,t_2)\subsetneq (-\varepsilon, T)$,  as above we conclude
$$\tilde{C}\dint_{B_{r}(0)}|x|^{-\frac{N-2s}{2}+\alpha} u_0(x)\,dx\leq \int_\Omega{u_0\varphi(x,0)\,dx}<+\infty.$$
\end{proof}
Conversely, we would like to find the optimal summability conditions on $g$ and $u_0$ to prove existence of weak solution. In this direction, notice that if $g\in L^2(0,T;H^{-s}(\Omega))$ and $u_0\in L^2(\Omega)$, by Remark \ref{coerciveoperator} we can assure the existence of an energy solution of the problem  \eqref{problemBG} whether          {$\lambda< \Lambda_{N,s}$, and in $H(\Omega)$ (see \eqref{Hardynorm}) for $\lambda=\Lambda_{N,s}$}. A sharper result, for a more general class of data, is the following.
\begin{Theorem}\label{suficiente}
Assume $0<\l\le \Lambda_{N,s}$, and that $g$ and $u_0$ satisfy
$$\int_\Omega \frac{u_0}{|x|^\gamma}\,dx<+\infty,\quad \int_0^T\int_\Omega \frac{g}{|x|^\gamma}dx\,dt<+\infty,$$
 where $\gamma$ was defined in \eqref{g1}. Then problem \eqref{problemBG} has a positive weak solution.
\end{Theorem}

\begin{proof}
Consider the approximated problems
\begin{equation}\label{approxBG}
\left\{
\begin{array}{rcl}
u_{nt}+(-\Delta)^s u_{n}&=&\lambda\dfrac{u_{n-1}}{|x|^{2s}+\frac{1}{n}}+g_n\inn \Omega\times (0,T),\\
u_{n}(x,t)&>&0\inn \Omega\times (0,T),\\
u_{n}(x,t)&=&0\inn (\ren\setminus\Omega)\times[ 0,T),\\
u_n(x,0)&=&T_{n}({u_0(x)})\mbox{ if }x\in \Omega,\\
\end{array}
\right.
\end{equation}
where
\begin{equation}
\left\{
\begin{array}{rcl}
u_{0t}+(-\Delta)^{s} u_{0}&=&g_1\inn \Omega\times (0,T),\\
u_{0}(x,t)&>&0\inn \Omega\times (0,T),\\
u_{0}(x,t)&=&0\inn (\ren\setminus\Omega)\times[ 0,T),\\
u_0(x,0)&=&T_{1}({u_0(x)})\mbox{ if }x\in \Omega,\\
\end{array}
\right.
\end{equation}
with $g_n= T_{n}(g)$ and
$$T_n(g)=
\begin{cases}g\hbox{ if }|g|\le n,\\
n\dfrac{g}{|g|}\hbox{ if }|g|>n.
\end{cases}$$
By Lemma \ref{DebilComparison}, it follows
that $u_0\leq u_1\leq \cdots\leq u_{n-1}\leq u_n \inn \ren \times (0,T)$.
Note that, since the right hand sides of these problems are bounded, every $u_n$ is actually an energy solution.

Consider  $\varphi$ the solution of the problem
\begin{equation}\label{varphiprob}
\left\{
\begin{array}{rcl}
-\varphi_t+(-\Delta)^s\varphi-\lambda\dfrac{\varphi}{|x|^{2s}}&=&1\quad\hbox{ in }\Omega\times (  {-\varepsilon},T),\\
\varphi&>&0\quad\hbox{ in }\Omega\times (  {-\varepsilon},T),\\
\varphi&=&0\quad\hbox{ on }(\mathbb{R}^N\setminus\Omega)\times   {(-\varepsilon,T]},\\
{\varphi(x,T)}&=& C\quad\hbox{ in }\Omega,
\end{array}
\right.
\end{equation}
{where $C>0$}. As a consequence of the weak Harnack inequality, Theorem
\ref{harnack}, and Proposition \ref{boundednessV}, for any cylinder $B_r(0)\times [t_1,t_2]\subset
\Omega\times (-\varepsilon,T)$ we find $c_1,\,c_2>0$ such that
\begin{equation}\label{upBoundPhi}
\frac{c_1}{|x|^\gamma}\le \varphi(x,t)\le \frac{c_2}{|x|^\gamma}.
\end{equation}
Since $\varphi$ also belongs to $L^2(0,T;H_0^s(\Omega))$, we can
use it as a test function in \eqref{approxBG}. Thus,
\begin{eqnarray*}
\int_0^T\int_{\Omega}{(u_n)_t\varphi\,dx\,dt}+\int_0^T\int_{\mathbb{R}^N}{u_n(-\Delta)^s\varphi\,dx\,dt}&=&\lambda\int_0^T\int_{\Omega}{\frac{u_{n-1}\varphi}{|x|^{2s}+\frac{1}{n}}\,dx\,dt}+\int_0^T\int_{\Omega}{g_n\varphi\,dx\,dt}\\
&\leq&\lambda\int_0^T\int_{\Omega}{\frac{u_n\varphi}{|x|^{2s}}\,dx\,dt}+\int_0^T\int_{\Omega}{g_n\varphi\,dx\,dt}.
\end{eqnarray*}
     Integrating in time and applying \eqref{varphiprob} and \eqref{upBoundPhi}, we conclude that
\begin{eqnarray*}
C\int_{\Omega} u_n(x,T)dx +\int_0^T\int_{\Omega}{u_n\,dx\,dt}&\leq&\int_0^T\int_{\Omega}{g_n\varphi\,dx\,dt}+\int_{\Omega}{T_n(u_0(x))\varphi(x,0)\,dx}\\
&\leq& \int_0^T\int_{\Omega}{g\varphi\,dx\,dt}+\int_{\Omega}{u_0(x)\varphi(x,0)\,dx}\\
&\leq& C\int_0^T\int_{\Omega}{\frac{g}{|x|^\gamma}\,dx\,dt}+C\int_{\Omega}{\frac{u_0(x)}{|x|^\gamma}\,dx}\\
&<&+\infty,
\end{eqnarray*}
by hypotheses. Hence, since the sequence $\{u_n\}-{n\in \mathbb{N}}$ is increasing, we can define $u:=\lim_{n\rightarrow\infty}u_n$, and conclude that $u\in L^1(\Omega\times(0,T))$
by applying the Monotone Convergence Theorem.

     Notice that, using the same computations as
above and integrating in $\Omega\times [0,t]$ with $t\le T$,
by considering the estimates on $\{u_n\}_{n\in\mathbb{N}}$, we
reach that
\begin{equation}\label{suppp}
\sup_{t\in [0,T]}\int_{\Omega} u_n(x,t)dx
+\int_0^T\int_{\Omega}{u_n\,dx\,dt}\le C\,  \mbox{
 for all   } n.
\end{equation}
Fix $T_1>T$, and define $\tilde{\varphi}$ as the unique
solution to the problem
\begin{equation}\label{bbo}
\left\{
\begin{array}{rcl}
-\tilde{\varphi}_t+(-\Delta)^s\tilde{\varphi}&=&1\quad\hbox{ in }\Omega\times ({-\varepsilon},T_1),\\
\tilde{\varphi}&>&0\quad\hbox{ in }\Omega\times ({-\varepsilon},T_1),\\
\tilde{\varphi}&=&0\quad\hbox{ on }(\mathbb{R}^N\setminus\Omega)\times {(-\varepsilon,T_1]},\\
\varphi(x,T_1)&=& 0\hbox{ in }\Omega.
\end{array}
\right.
\end{equation}
It is clear that $\tilde{\varphi}\in L^\infty(\Omega\times (-\varepsilon,{T_1}))$ and
$\tilde{\varphi}(x,t)\ge \bar{C}>0$ for all $(x,t)\in B_r(0)\times
[0,T]$, where $B_r(0)\subset \subset \Omega$. Now, using
$\tilde{\varphi}$ as a test function in \eqref{approxBG} and
integrating in $\Omega\times (0,T)$, it follows that
$$
\int_{\Omega}u_n(x,T)\tilde{\varphi}(x,T)\,dx\,dt+\int_0^T\int_{\Omega}{u_n\,dx\,dt}\geq
\lambda\int_0^T\int_{\Omega}\frac{u_{n-1}\tilde{\varphi}}{|x|^{2s}+
\frac{1}{n}}\,dx\,dt.
$$
Thus
$$
\lambda\int_0^T\int_{\Omega}\frac{u_{n-1}\tilde{\varphi}}{|x|^{2s}+
\frac{1}{n}}\,dx\,dt\le C\sup_{\{t\in [0,T]\}}\int_{\Omega}
u_n(x,t)dx +\int_0^T\int_{\Omega}{u_n\,dx\,dt}\le C\, \;\mbox{
 for all   }\; n.
$$
Hence
\begin{eqnarray*}
&\dyle \int_0^T\int_{\Omega}\frac{u_{n-1}}{|x|^{2s}+
\frac{1}{n}}\,dx\,dt=\dyle
\int_0^T\int_{B_r(0)}\frac{u_{n-1}}{|x|^{2s}+
\frac{1}{n}}\,dx\,dt+\dyle \int_0^T\int_{\Omega\backslash
B_r(0)}\frac{u_{n-1}}{|x|^{2s}+
\frac{1}{n}}\,dx\,dt\\
&\le C\dyle
\int_0^T\int_{\Omega}\frac{u_{n-1}\tilde{\varphi}}{|x|^{2s}+
\frac{1}{n}}\,dx\,dt+C\dyle
\int_0^T\int_{\Omega}u_{n-1}\,dx\,dt\le C.
\end{eqnarray*}
Therefore, by the Monotone Convergence Theorem we conclude that
$$
\dfrac{u_{n-1}}{|x|^{2s}+\frac{1}{n}}+g_n\uparrow
\dfrac{u}{|x|^{2s}}+g\mbox{  strongly in }L^1(\Omega\times (0,T)).
$$
To conclude that $u$ is a weak solution to problem \eqref{problemBG},
it remains to check that  $u\in \mathcal{C}([0,T);L^1(\Omega))$. We claim that $\{u_n\}_{n\in\mathbb{N}}$ is a Cauchy sequence in $\mathcal{C}([0,T]; L^{1}(\Omega))$, and hence the result follows. In order to prove this, we closely follow the
arguments in \cite{PR}.

For $n,m\in \mathbb{N}$, such that $n\geq
m$,  denote $u_{n,m}:= u_n-u_m$, and $g_{n,m}:= g_n-g_m$. Clearly,  $u_{n,m}, g_{n,m}\ge 0$.
We set
$$
C_{         {n,m}}:= \lambda\dint_{0}^{t}\!\!\!\dint_{\Omega}
\dfrac{u_{n,m}}{|x|^{2s}}\,dx\,d\tau+  {\int_0^t
\dint_{\Omega} g_{n,m}T_1(u_{n,m})\,dx\,d\tau},
$$
and then $C_{         {n,m}}\to 0$ as $n,m\to \infty$.

By the definition of the approximated problems in \eqref{approxBG}
and the linearity of the operator, for $t\leq T$,
$$
\dint_{0}^{t}\!\!\! \dint_{\Omega} (u_{n,m})_{t}
T_{1}(u_{n,m})\,dx\,d\tau+\dint_{0}^{t}\!\!\! \dint_{\Omega}
(-\Delta)^{s} (u_{n,m}) T_{1}(u_{n,m})\,dx\,d\tau\leq C_{         {n,m}}.
$$
Since $u_{n,m}\in L^{2}(0,T; H_{0}^{s}(\Omega))$, it follows that (see \cite{LPPS} for a detailed proof)
$$
\dint_{0}^{t}\!\!\! \dint_{\Omega} (-\Delta)^{s} (u_{n,m}) T_{1}(u_{n,m})\,dx\,d\tau \geq \dint_{0}^{t} \| T_{1}(u_{n,m})\|_{H_{0}^{s}(\Omega)}^{2}\,d\tau\geq 0,
$$
and therefore,
$$
\dint_{0}^{t}\!\!\! \dint_{\Omega} (u_{n,m})_{t} T_{1}(u_{n,m})\,dx\,d\tau\leq C_{n,m}.
$$
Let us define
$$
\Psi(s):=\int_0^s T_1(\sigma) d\sigma.
$$
Since $u_n\in \mathcal{C}([0,T]; L^{2}(\Omega))$, then $$ \dint_{0}^{t}\!\!\!
\dint_{\Omega} (u_{n,m})_{t} T_{1}(u_{n,m})\,dx\,d\tau =
\dint_{\Omega} (\Psi (u_{n,m})(t)- \Psi (u_{n,m})(0))\,dx. $$

     Thus
$$
\dint_{\Omega} \Psi (u_{n,m})(t)\,dx \le C_{n,m}+\dint_{\Omega}
\Psi (u_{n,m})(0)\,dx.
$$
Taking into account that  $\Psi (u_{n,m})(0)=\Psi (T_n(u_0)-T_m(u_0))$ and by noticing that $\Psi(s)\le |s|$ and $T_n(u_0)-T_m(u_0)\to 0$
strongly in $L^1(\O)$ as $n,m\to \infty$, we obtain that
$$
\dint_{\Omega} \Psi (u_{n,m})(t)\,dx \to 0 \mbox{  as  }n,m\to
\infty,\;\;\mbox{         { uniformly in $t$}}.
$$
Now, since
$$
\dint_{|u_{n,m}|<1}\, \frac{|u_{n,m}|^{2}(t)}{2}\,dx + \dint_{|u_{n,m}|>1}\,
|u_{n,m}| (t)\,dx\leq \dint_{\Omega} \Psi
(u_{n,m})(t)\,dx, $$
         {we conclude that $u_{n,m}\rightarrow 0$ uniformly in $t$.}

Thus $\{u_n\}_{n\in\mathbb{N}}$ is a Cauchy sequence in $\mathcal{C}([0,T];
L^{1}(\Omega))$ and passing to the limit in the weak formulation
of the approximated problems, one obtains that $u$ is a positive
weak solution of problem \eqref{problemBG} in $\Omega\times
(0,T)$.
\end{proof}

Next, we see that $\Lambda_{N,s}$ provides a real restriction on $\lambda$.
\begin{Proposition}\label{largeLambda}
If $\l> \Lambda_{N, s}$, problem \eqref{problemBG} has no positive weak supersolution.
\end{Proposition}

\begin{proof}
Consider $\tilde{u}$ as a weak supersolution
to the problem
\begin{equation}\left\{
\begin{array}{rcl}
u_t+(-\Delta)^s u-\Lambda_{N,s}\dfrac{u}{|x|^{2s}}&= &(\l-\Lambda_{N,s})\dfrac{u}{|x|^{2s}}+g \inn\Omega\times (0,T),\\
u(x,t)&>&0\inn \Omega\times (0,T),\\
u(x,t)&=&0\inn (\ren\setminus\Omega)\times[ 0,T).\\
\end{array}\right.
\end{equation}
Since in the left hand side the constant is $\Lambda_{N,s}$, we are in the case $\alpha=0$, and by Theorem \ref{necesario}, necessarily
$$\left((\l-\Lambda_{N,s})\dfrac{\tilde{u}}{|x|^{2s}}+g\right)|x|^{-\frac{N-2s}{2}}\in L^1(B_r(0)\times (t_1,t_2)),$$
for any $B_r(0)\times (t_1,t_2)\subset\subset\Omega\times(0,T)$ small enough. In particular, this implies
$$(\l-\Lambda_{N,s})\dfrac{\tilde{u}}{|x|^{2s}}|x|^{-\frac{N-2s}{2}}\in L^1(B_r(0)\times (t_1,t_2)),$$
and hence, applying          {Corollary \ref{coroHar}} again,
$$(\l-\Lambda_{N,s})|x|^{-N}\in L^1(B_r(0)\times (t_1,t_2)),$$
what is a contradiction. Therefore, there does not exist a positive supersolution if $\l> \Lambda_{N, s}$.
\end{proof}
\begin{remark}
The previous nonexistence result implies that for $\lambda>\Lambda_{N,s}$ an instantaneous and complete blow up phenomena occurs. The proof is a simple adaptation of Theorem \ref{th:blowup}, where this result is proved for a more involved semilinear problem.
\end{remark}

Furthermore, we can state a nonexistence result that shows the
optimality of the power $p=1$ in the singular term
$\dfrac{u^p}{|x|^{2s}}$. The proof for this nonlocal problem
closely follows the classical case, due to H. Brezis and X.
Cabr\'{e} (see \cite{BC}).
\begin{Theorem}
Let $p>1$, and let $u\ge 0$ satisfy
$$u_t+(-\Delta)^su\ge \frac{u^p}{|x|^{2s}}\hbox{ in }\Omega\times (0,T),$$
in the weak sense. Then $u\equiv 0$.
\end{Theorem}
\begin{proof}
Consider a cylinder $B_\tau(0)\times (t_1,t_2)$. If $u\gneq 0$, by the Maximum Principle (Theorem \ref{SMP}), we know that there exists $\varepsilon>0$ so that
$$u\geq\varepsilon \hbox{ in }B_\tau(0)\times (t_1,t_2).$$
Let define
$$\phi(s)=
\begin{cases}
\dfrac{1}{(p-1)\varepsilon^{p-1}}-\dfrac{1}{(p-1)s^{p-1}}\hbox{ if }s\geq\varepsilon,\\
\,\\
\dfrac{1}{\varepsilon^p}(s-\varepsilon)\hbox{ if }s<\varepsilon.
\end{cases}$$
Notice that $0\leq \phi<+\infty$ in $[\varepsilon,+\infty)$, $\phi(\varepsilon)=0$, $\phi'(\varepsilon)=0$, and $\phi$ is a
$\mathcal{C}^1$ function satisfying $\phi'(s)=\dfrac{1}{s^p}$ for $s\geq\varepsilon$. Moreover, since $\phi$ is concave,  it follows that $(-\Delta)^s(\phi(u))\geq \phi'(u)(-\Delta)^su$ and thus,
$$(\phi(u))_t+(-\Delta)^s(\phi(u))\geq \phi'(u)\left(u_t+(-\Delta)^su\right)\geq \frac{1}{|x|^{2s}} \hbox{ in  }B_\tau(0)\times
(t_1,t_2)$$ with $u\geq\varepsilon.$

Without loss of generality we can assume that $\tau=1$. Define
$w(x,t)=(t-t_1)\vartheta(x)$ where
$$
\vartheta (x)= \left\{
\begin{array}{lll}
\log\left(\dfrac{1}{|x|}\right) &\mbox{  if   } &|x|<1,\\
0 &\mbox{  if   }&|x|\ge 1,
\end{array}
\right.
$$
then $w(x,t_1)=0$ in $B_1(0)$ and  $w(x,t)=0$ in $\ren\backslash
B_1(0)$.

We claim that
$$
w_t+(-\Delta)^s w\le \dfrac{C}{|x|^{2s}} \mbox{  in }B_1(0)\times
(t_1,t_2)
$$
where $C\equiv C(t_1,t_2)>0$. Indeed, we have
$$
w_t+(-\Delta)^sw=\vartheta(x)+(t-t_1)(-\Delta)^s\vartheta.$$ It is
clear that $\vartheta(x)\le \dfrac{C_1}{|x|^{2s}}\mbox{  in
}B_1(0)\times (t_1,t_2)$, and hence to prove the claim we have to show that
$$
(-\Delta)^s\vartheta(x)\le \dfrac{C_2}{|x|^{2s}}\mbox{ for all
}x\in B_1(0).
$$
In fact,
\begin{eqnarray*}
(-\Delta)^s\vartheta(x) &= &\int_{\ren}
\dfrac{(\vartheta(x)-\vartheta(y))}{|x-y|^{N+2s}}dy\\
&=& \int_{\{|y|<1\}}
\dfrac{(\vartheta(x)-\vartheta(y))}{|x-y|^{N+2s}}dy +\int_{\{|y|>1\}}
\dfrac{\vartheta(x)}{|x-y|^{N+2s}}dy\\
&= & I_1(x)+I_2(x).
\end{eqnarray*}
We closely follow the arguments  in \cite{FV} to estimate the integrals above.
By setting $r:=|x|$ and $\rho:=|y|$, then $x=rx', y=\rho y'$ where
$|x'|=|y'|=1$. Thus,
$$
I_1(x)=\dyle \int_0^1\log(\frac{\rho}{r})\rho^{N-1}\left(
\dint\limits_{|y'|=1}\dfrac{dH^{N-1}(y')}{|r x'-\rho y'|^{N+2s}}
\right) \,d\rho.
$$
Calling $\sigma:=\dfrac{\rho}{r}$, then
$$
I_1(x)=\frac{D_1(|x|)}{|x|^{2s}},
$$
where $$D_1(r)=\dint\limits_0^{\frac{1}{r}}
\log(\sigma)\sigma^{N-1}K(\sigma)\,d\sigma,
$$
and
$$ K(\sigma):=\dint\limits_{|y'|=1}\dfrac{dH^{N-1}(y')}{|x'-\sigma
y'|^{N+2s}}=2\frac{\pi^{\frac{N-1}{2}}}{\Gamma(\frac{N-1}{2})}\int_0^\pi
\frac{\sin^{N-2}(\eta)}{(1-2\sigma \cos
(\eta)+\sigma^2)^{\frac{N+2s}{2}}}d\eta.
$$
In the same way we have
$$
I_2(x)=\frac{D_2(|x|)}{|x|^{2s}} \mbox{  where    }
D_2(r)=-\log(r)\dint\limits_{\frac{1}{r}}^{+\infty}
\sigma^{N-1}K(\sigma)\,d\sigma.$$ Combining the estimates above,
we get
$$
(-\Delta)^s\vartheta(x)=\frac{D(|x|)}{|x|^{2s}}
$$
with
$$
D(r):= D_1(r)+D_2(r)=\dint\limits_0^{\frac{1}{r}}
\log(\sigma)\sigma^{N-1}K(\sigma)\,d\sigma-\log(r)\dint\limits_{\frac{1}{r}}^{+\infty}
\sigma^{N-1}K(\sigma)\,d\sigma.
$$
Notice that $K(\sigma)\le C|1-\sigma|^{-1-2s}$ as $\sigma \to 1$
and $K\left(\frac{1}{\sigma}\right)=\sigma^{N+2s}K(\sigma)$ for all
$\sigma>0$.

If $s\le \frac 12$, then using the behavior of $K$ at $+\infty$ we can
easily prove that $|D_1(r)|+|D_2(r)|\le C$ for all $r\le 1$. To study the general case we need to do some sharp computations.
Since
$$
D(r)\le \dint\limits_0^{+\infty}
\log(\sigma)\sigma^{N-1}K(\sigma)\,d\sigma=\bar{D}, $$ then to
finish we have to show that $|\Bar{D}|<\infty$.
Notice that
$$
\bar{D}=\dint\limits_0^{1}
\log(\sigma)\sigma^{N-1}K(\sigma)\,d\sigma+\dint\limits_1^{+\infty}
\log(\sigma)\sigma^{N-1}K(\sigma)\,d\sigma.$$
By putting $\theta:=\dfrac{1}{\sigma}$ in the first integral, and by using
the fact that $K\left(\dfrac{1}{\theta}\right)=\theta^{N+2s}K(\theta)$, it
follows that
\begin{eqnarray*}
\bar{D}&= &
\dint\limits_1^{+\infty}K(\sigma)\log(\sigma)(\sigma^{N-1}-\sigma^{2s-1})d\sigma.
\end{eqnarray*}
Now, due to the behavior of $K$ near $1$ and
$\infty$ we obtain that $0<\bar{D}<\infty$. Thus
$$
(-\Delta)^s\vartheta(x)\le \frac{\bar{D}}{|x|^{2s}}.
$$
Hence
$$
w_t+(-\Delta)^s w\le \dfrac{C}{|x|^{2s}} \mbox{  in }B_1(0)\times
(t_1,t_2),
$$
where $C=C_1+(t_2-t_1)\bar{D}$ and the claim follows.

Fixed $\e>0$  such that $\e C<<1$,  we obtain that
$$(\phi(u)-\e w)_t+(-\Delta)^s(\phi(u)-\e w)\geq 0\hbox{ in }B_1(0)\times (t_1,t_2),$$
in the weak sense. Since $(\phi(u)-\e w)(x,t_1)\ge 0$ in $B_1(0)$
and $(\phi(u)-\e w)(x,t_1)\ge 0$ in $\ren \backslash B_1(0)\times
(t_1,t_2)$, then the comparison principle implies that $\phi(u)-\e
w\geq 0$ in $B_\tau(0)\times (t_1,t_2)$. Since $w$ is unbounded,
we reach a contradiction with the fact that $\phi$ is a bounded
function, and the proof is finished.
\end{proof}

Thus, as a straightforward consequence we obtain the following result.

\begin{Corollary}
Let $g\in L^1(\Omega\times(0,T))$, $g\ge 0$, and $p>1$. Therefore, the problem
\begin{equation*}
\left\{
\begin{array}{rcl}
 u_t+(-\Delta)^{s} u&=&\l\dfrac{\,u^p}{|x|^{2s}}+g\mbox{ in } \Omega\times (0,T),\\
u(x,t)&=&0\inn (\ren\setminus\Omega)\times[ 0,T),\\
u(x,0)&=&u_0(x)\gneq 0 \mbox{ if }x\in\O,
\end{array}
\right.
\end{equation*}
has no positive weak solution.
\end{Corollary}

\section{Existence and nonexistence results for a semilinear problem}\label{5}
     The goal of this section is to study how the addition of a semilinear term of the form $u^p$, with $p>1$, interferes with the solvability of the previous problems. As in the classical heat equation, see \cite{APP}, the relevant feature is that for every $0<\lambda<\Lambda_{N,s}$ there exists a threshold for the existence, $p_{+}(\lambda,s)$, that depends on the spectral parameter. Indeed, we will consider  the problem
\begin{equation}\label{semilinearprob}
\left\{
\begin{array}{rcl}
 u_t+(-\Delta)^{s} u&=&\l\dfrac{\,u}{|x|^{2s}}+u^p+          {f}\mbox{ in } \Omega\times (0,T),\\
u(x,t)&>&0\inn \Omega\times (0,T),\\
u(x,t)&=&0\inn (\ren\setminus\Omega)\times[ 0,T),\\
u(x,0)&=&u_0(x) \mbox{ if }x\in\O,
\end{array}
\right.
\end{equation}
with $p>1$ and          {$0<\lambda<\Lambda_{N,s}$}. By weak or energy solutions of this problem, we mean solutions in the sense of Definition \ref{veryweak} and Definition \ref{energysol} by fixing $F=\l\dfrac{\,u}{|x|^{2s}}+u^p+ c f$.

     We will prove that there exists such critical exponent $p_{+}(\lambda,s)$ so that one can prove existence of solution for problem \eqref{semilinearprob} whether $1<p<p_+(\lambda,s)$, and nonexistence for $p> p_+(\lambda,s)$. Following the same ideas as in \cite{BMP}, one can expect $p_{+}(\lambda,s)$ to depend on $s$ and $\lambda$, and in particular to satisfy
$$
p_{+}(\lambda,s)= 1+\dfrac{2s}{\frac{N-2s}{2}-\alpha}=1+\dfrac{2s}{\gamma}.
$$
Note that if $\lambda=\Lambda_{N,s}$, namely, $\alpha=0$, then $p_{+}(\lambda,s)= 2^{*}_{s}-1$, and if $\lambda=0$, i.e., $\alpha=\dfrac{N-2s}{2}$, then $p_{+}(\lambda,s)=\infty$.

We will need some auxiliary results that allow us to build a
solution whenever we have a supersolution. To prove existence of a
weak solution to \eqref{semilinearprob} with $L^1$ data from a
weak supersolution,   we will consider  the {\it solution obtained
as limit of solutions of approximated problems} (see  for instance \cite{DaA}  in the local parabolic operators case).

\begin{Lemma}\label{Comparison}
If $\bar{u}\in \mathcal{C}([0,T); L^{1}(\Omega))$ is a weak positive
supersolution to the equation in \eqref{semilinearprob} with $\lambda\leq
\Lambda_{N,s}$ and $f\in L^1(\Omega\times (0,T))$, then there exists a positive minimal weak solution to problem \eqref{semilinearprob} obtained as limit of solutions of approximated problems.
\end{Lemma}
\begin{proof} If $\bar{u}$ is a positive supersolution to \eqref{semilinearprob} with
$\lambda\leq \Lambda_{N,s}$, we construct a sequence $\{u_n\}_{n\in\mathbb{N}}$
starting with
\begin{equation}\label{eq:v0}
\left\{
\begin{array}{rcl}
u_{0t}+(-\Delta)^{s} u_{0}&=&T_{1}(f)\inn \Omega\times (0,T),\\
u_{0}(x,t)&>&0\inn \Omega\times (0,T),\\
u_{0}(x,t)&=&0\inn (\ren\setminus\Omega)\times[ 0,T),\\
u_0(x,0)&=&T_{1}({u_0(x)})\mbox{ if }x\in \Omega.\\
\end{array}
\right.
\end{equation}
By the  Weak Comparison Principle (Lemma \ref{DebilComparison}), it follows
that $u_0\leq \bar{u}\inn \ren \times (0,T)$. By iteration we define
\begin{equation}\label{eq:vn}
\left\{
\begin{array}{rcl}
u_{nt}+(-\Delta)^s u_{n}&=&\lambda\dfrac{u_{n-1}}{|x|^{2s}+\frac{1}{n}}+ u_{n-1}^{p}+T_{n}(f)\inn \Omega\times (0,T),\\
u_{n}(x,t)&>&0\inn \Omega\times (0,T),\\
u_{n}(x,t)&=&0\inn (\ren\setminus\Omega)\times[ 0,T),\\
u_n(x,0)&=&T_{n}({u_0(x)})\mbox{ if }x\in \Omega.\\
\end{array}
\right.
\end{equation}
In fact, $\{u_n\}_{n\in \mathbb{N}}\subset
\mathcal{C}([0,T); L^{{1}}(\Omega))\cap
L^p([0,T);L^{p}(\Omega))$ (see \cite{LPPS}).
As above it follows that $u_{0}\leq \dots\leq u_{n-1}\leq u_{n}\leq
\bar{u}\inn \ren\times (0,T)$, so we obtain the pointwise limit
$u:=\lim u_n$ that verifies $u\le \bar{u}$ and
\begin{equation}\label{eq:v}
\left\{
\begin{array}{rcl}
u_{t}+(-\Delta)^{s} u&=&\lambda\dfrac{u}{|x|^{2s}}+u^{p}+f\inn \Omega\times (0,T),\\
u(x,t)&>&0\inn \Omega\times (0,T),\\
u(x,t)&=&0\inn (\ren\setminus\Omega)\times[ 0,T),\\
u(x,0)&=& {u_0(x)} \mbox{ if }x\in \Omega,
\end{array}
\right.
\end{equation}
in the weak sense.          {The fact that $u\in\mathcal{C}([0,T);L^1(\Omega))$ follows as in the proof of Theorem \ref{suficiente}.}
 \end{proof}
Likewise, if the supersolution belongs to the energy space, the solution we find will be also an energy solution.
\begin{Lemma}\label{ApproxEnergy}
If $\bar{u}\in L^2(0,T; H_0^s(\Omega))$ with $\bar{u}_t\in L^{2}(0,T; H^{-s}(\O))$ is a positive finite energy
supersolution to \eqref{semilinearprob} with $\lambda\leq
\Lambda_{N,s}$ and $f\in L^2(0,T;H^{-s}(\Omega))$, then there exists a positive minimal energy solution to problem \eqref{semilinearprob} obtained as limit of solutions of the approximated problems.
\end{Lemma}
\begin{proof}
Proceeding as in the proof of Lemma \ref{Comparison} and using the Comparison Principle for energy solutions (Lemma \ref{Comparison Principle}), we can build a sequence $\{u_n\}_{n\in\mathbb{N}}$ of energy solutions of the approximated problems \eqref{eq:vn}, so that,
$$u_0\leq u_1\leq ...\leq u_n\leq...\leq \bar{u} \quad\hbox{ in }\mathbb{R}^N\times (0,T).$$

     Hence, by the Monotone Convergence Theorem we can define $u:=\lim_{n\rightarrow \infty}u_n\leq \bar{u}$. Moreover, applying the energy formulation of $u_n$,
\begin{eqnarray*}
\|u_n\|^2_{L^2(0,T;H_0^s(\Omega))}&=&\int_0^T{\|u_n(\cdot,t)\|^2_{H_0^s(\Omega)}\,dt}=\int_0^T{\int_Q{\frac{(u_n(x,t)-u_n(y,t))^2}{|x-y|^{N+2s}}\,dx\,dy}\,dt}\\
&=&\frac{2}{a_{N,s}}\left\{\int_0^T\int_\Omega\left(\lambda\frac{u_n^2}{|x|^{2s}}+u_n^{p+1}+T_n(f)u_n\right)dx\,dt-\int_0^T\int_\Omega(u_n)_tu_n\,dx\,dt\right\}\\
&=& \frac{2}{a_{N,s}}\left\{\int_0^T\int_\Omega\left(\lambda\frac{u_n^2}{|x|^{2s}}+u_n^{p+1}+T_n(f)u_n\right)dx\,dt-{\frac 12}\int_\Omega u_n(x,T)^2\,dx\right.\\
&&\left.+{\frac 12}\int_\Omega u_n(x,0)^2\,dx\right\}\\
&\leq&\frac{2}{a_{N,s}}\left\{\int_0^T\int_\Omega\left(\lambda\frac{\bar{u}^2}{|x|^{2s}}+\bar{u}^{p+1}+f\bar{u}\right)dx\,dt+{\frac 12}\int_\Omega \bar{u}(x,0)^2\,dx\right\}\\
&\leq& C.
\end{eqnarray*}

     Thus, up to a subsequence, we know that $u_n\rightharpoonup u$ in $L^2(0,T;H_0^s(\Omega))$. Likewise, for every $0\leq t \leq T$,
\begin{eqnarray*}
\|(u_n)_t\|_{H^{-s}(\Omega)}&=&\sup_{\|\varphi\|_{H_0^s(\Omega)}\leq 1}\bigg|\int_\Omega{(u_n)_t\varphi\,dx}\bigg|\\
&\leq& \sup_{\|\varphi\|_{H_0^s(\Omega)}\leq 1}\bigg|\int_Q\frac{(u_n(x)-u_n(y))(\varphi(x)-\varphi(y))}{|x-y|^{N+2s}}\,dx\,dy\bigg|\\
& &+\sup_{\|\varphi\|_{H_0^s(\Omega)}\leq 1}\left\{\int_\Omega{u_n^p\varphi\,dx}+\lambda\int_\Omega{\frac{u_n\varphi}{|x|^{2s}}\,dx}+\int_\Omega{T_n(f)\varphi}\right\}\\
&\leq&\sup_{\|\varphi\|_{H_0^s(\Omega)}\leq 1}\left\{\|u_n\|_{H_0^s(\Omega)}\|\varphi\|_{H_0^s(\Omega)}+\int_\Omega{\bar{u}^p\varphi\,dx}+\lambda\int_\Omega{\frac{\bar{u}\varphi}{|x|^{2s}}\,dx}+\int_\Omega{f\varphi}\,dx\right\}\\
&\leq& C(\|u_n\|_{H_0^s(\Omega)}{+1}+\|f\|_{L^2(\Omega)}).
\end{eqnarray*}
Hence,
\begin{equation*}
\int_0^T{\|(u_n)_t\|_{H^{-s}(\Omega)}^2\,dt}\leq C(\|u_n\|^2_{L^2(0,T;H_0^s(\Omega))}{+1}+\|f\|_{L^2(0,T;L^2(\Omega))}^2)\leq C,
\end{equation*}
and therefore, up to a subsequence, $(u_n)_t\rightharpoonup u_t$ in $L^2(0,T;H^{-s}(\Omega))$,
and we can pass to the limit to conclude that $u$ is a finite energy solution to \eqref{semilinearprob}.

\end{proof}

\subsection{Nonexistence results for $p> p_+(\lambda,s)$. Instantaneous and complete blow up.}\label{51}
     Assume first that $p$ is greater than the threshold exponent $p_+(\lambda,s)$. Thus, we can formulate the nonexistence result as follows.
\begin{Theorem}\label{nonexistence}
Let $\lambda\leq \Lambda_{N,s}$. If $p> p_+(\lambda,s)$, then problem \eqref{semilinearprob} has no positive weak supersolution. In the case where $f\equiv 0$,
the unique nonnegative supersolution is $u\equiv 0$.
\end{Theorem}
\begin{proof}
Without loss of generality, we can assume $f\in L^{\infty}(\Omega\times (0,T))$.
We argue by contradiction. Assume that $\tilde{u}$ is a positive weak supersolution of \eqref{semilinearprob}. Then $\tilde{u}_t+(-\Delta)^s
\tilde{u}-\l\dfrac{\tilde{u}}{|x|^2}\geq 0$ in $\Omega\times (0,T)$ in the weak sense.

     Since $\tilde{u}$  is also a weak supersolution in any $B_{R}(0)\times (T_{1}, T_2)\subset\subset\Omega\times (0,T)$, then  by  Lemma \ref{Comparison}, the  problem
\begin{equation}\label{eq:prob*}
\left\{
\begin{array}{rcl}
 u_t+(-\Delta)^{s} u&=&\l\dfrac{\,u}{|x|^{2s}}+u^p+f\mbox{ in } B_{R}(0)\times (T_{1}, T_2),
 \\
u(x,t)&>&0\inn  B_{R}(0)\times (T_1,T_2),\\
u(x,t)&=&0\inn (\ren\setminus B_{R}(0))\times[T_1,T_2),\\
u(x,T_1)&=&\tilde{u}(x,T_1)
 \mbox{ if }x\in  B_{R}(0),
\end{array}
\right.
\end{equation}
has a minimal solution $u$ obtained by approximation of truncated problems in
$B_{R}(0)\times (T_1,T_2)$. In particular  $u=
\lim u_{n}$, with          {$u_{n}\in L^{\iy}(B_{R}(0)\times (T_1,T_2))$ being the energy solution to \eqref{eq:vn}} in  $B_{R}(0)\times (T_1,T_2)$.\\
Notice that ${u}_t+(-\Delta)^{s} {u}-\l\dfrac{{u}}{|x|^{2s}}\geq
0\inn B_{r_1}(0)\times (T_1,T_2)$ in the weak sense, and
therefore, by           {Corollary \ref{coroHar}},  for any cylinder $B_{r}(0)\times (t_{1}, t_2)$,
with $0<r<r_1<R$, $0<T_1<t_1<t_2<T_2\leq T$ there exists a constant
$C=C(N,r_1, t_1, t_2)$ such that ${u}\geq C
|x|^{-\frac{N-2s}{2}+\alpha}$ and then for $r$ small enough $u>1$
in $B_{r}(0)\times (T_{1}, T_2)$.

     In particular, since $u\in
L^1(\Omega\times (0,T))$, then using the fact that $u\ge 1$ in
$B_{r}(0)\times (T_{1}, T_2)$, we reach that $\log(u)\in L^p(
B_{r}(0)\times (t_{1}, t_2))$,
 for all $p\ge 1$.
By a suitable scaling, we can assume that the cylinder is
$B_{r}(0)\times (0,\tau)$.

Let
$\phi\in \mathcal{C}^{\infty}_0(B_r(0))$, then using
$\dfrac{|\phi|^{2}}{u_n}$  as a test function in the approximated
problems \eqref{eq:vn} and applying the Picone (Theorem \ref{Picone}) and Sobolev (Theorem \ref{Sobolev}) inequalities,
\begin{equation*}\begin{split}
\dint_{0}^{\tau}\dint_{B_r(0)} u_n^{p-1}\phi^2\,dx\,dt&\leq
\dint_0^\tau\dint_{B_r(0)} \dfrac{|\phi|^{2}}{u_n}u_{n t}dxdt+\dint_0^\tau\dint_{B_r(0)} (- \Delta)^{s} u_n\dfrac{|\phi|^{2}}{u_n}dx\,dt\\\
&\le\dint_{B_r(0)}|\log u_n(x,\tau)|
\phi^2\,dx+ C'(N,s,\tau)\|\phi\|^2_{H^{s}_{0}(\Omega)}.
\end{split}\end{equation*}
Therefore, passing to the limit as $n$ tends to infinity, and
 considering that ${u}\geq C
|x|^{-\frac{N-2s}{2}+\alpha}$ in $B_r(0)\times (0,\tau)$, we
obtain
$$
\begin{array}{lcl}
\dint_{B_r(0)}|\log u(x,\tau)|
\phi^2\,dx+C' (N,s,\tau)\|\phi\|^2_{H^{s}_{0}(\Omega)}&\geq&
\dint_{0}^{\tau}\dint_{B_r(0)} u^{p-1}\phi^2\,dx\,dt\\
&\geq& C
\dint_{0}^{\tau}\dint_{B_r(0)}\dfrac{\phi^2}{|x|^{(p-1)(\frac{N-2s}{2}-\alpha)}}dx dt.
\end{array}
$$
Using H\"{o}lder and Sobolev inequalities, it follows that
$$
\begin{array}{rcl}
\dint_{B_r(0)}|\log (u(x,\tau))|
|\phi|^{2}dx&\leq&\left(\dint_{B_{r}(0)}|\phi|^{2_{s}^{*}}\,dx\right)^{\frac{2}{2_{s}^{*}}}\left(\dint_{B_{r}(0)}|\log
u(x,\tau)|^{\frac{N}{2s}}\,dx\right)^{\frac{2s}{N}}\\&\leq&
\left(\dint_{B_{r}(0)}|\log
u(x,\tau)|^{\frac{N}{2s}}\,dx\right)^{\frac{2s}{N}}{S}\|\phi\|^{2}_{H^{s}_{0}(\Omega)},
\end{array}
$$
where $S$ is the optimal constant in the Sobolev embedding.
Thus we have
$$\left[C'(N,s,\tau)+\left(\dint_{B_{r}(0)}|\log
u(x,\tau)|^{\frac{N}{2s}}\,dx\right)^{\frac{2s}{N}}{S}\right] \|\phi\|^{2}_{H^{s}_{0}(\Omega)}
\geq C
\dint_{B_r(0)}\dfrac{\phi^2}{|x|^{(p-1)(\frac{N-2s}{2}-\alpha)}}{dx}.
$$
Since $p>p_{+}(\lambda,s)$ then, for a cylinder small enough,
$(p-1)\left(\dfrac{N-2s}{2}-\alpha\right)>2s$ and we obtain a
contradiction with the Hardy inequality.
\end{proof}

The previous nonexistence result is very strong in the sense that a complete and instantaneous
blow up phenomenon occurs. That is, if $u_n$ is the solution to the approximated problems \eqref{eq:vn}, then $u_n(x,t)\to \infty$ as $n\to\infty$, where $(x,t)$ is an arbitrary point in $\Omega\times
(0,T)$.

\begin{Theorem}\label{th:blowup}
Let $ u_{n}$ be a solution to the problem \eqref{eq:vn} with
$p>p_+(\lambda,s)$. Then $u_{n}(x_{0},t_0)\rightarrow \iy$, for
all $(x_{0},t_0)\in\O\times(0,T)$.
\end{Theorem}
\begin{proof} Without loss of generality, we can assume that $\l\le \Lambda_N$.
The existence of a positive solution to problem
\eqref{eq:vn} is clear and, due to the Comparison Principle, we know that $u_n\le u_{n+1}$ for all $n\in\mathbb{N}$.

     Suppose by contradiction that there exists
$(x_0,t_0)\in\O\times (0,T)$ such that
$$u_{n}(x_0,t_0)\rightarrow C_0<\iy\,\,\mbox{  as  } n\to \infty.$$
By using the Harnack
inequality (see Lemma \ref{Harnack inequality}),  there exists $s_0>0$ and a positive constant $C=
C(N,s_0,t_0,\beta)$ such that
$$
\iint_{R_0^-}u_{n}(x,t)\,dx\,dt\leq C\mbox{ ess}
\inf\limits_{R_0^+}u_{n}\leq C,
$$
where $R_0^-=B_{s_0}(x_0)\times (t_0-\frac{3}{4}\beta,
t_0-\frac{1}{4}\beta) $ and $R_0^{+}=B_{s_0}(x_0)\times
(t_0+\frac{1}{4}\beta, t_0+\frac{3}{4}\beta)$.

     Without loss of generality, we can suppose $x_0=0$. Otherwise, we can find a finite sequence of points $\{x_i\}_{i=0}^M$, ending with $x_M=0$, and of radius $\{s_i\}_{i=0}^M$ such that $B_{s_i}(x_i)\subset\Omega$, $B_{s_i}(x_i)\cap B_{s_{i+1}}(x_{i+1})\neq\emptyset$, for all $i=0,\ldots,M$ and, by the Hanarck inequality,
$$
\iint_{R_i^-}u_{n}(x,t)\,dx\,dt\leq C\mbox{ ess}
\inf\limits_{R_i^+}u_{n},
$$
where $R_i^-=B_{s_i}(x_i)\times (t_i-\frac{3}{4}\beta,
t_i-\frac{1}{4}\beta) $ and $R_i^{+}=B_{s_i}(x_i)\times
(t_i+\frac{1}{4}\beta, t_i+\frac{3}{4}\beta)$, $t_i\in (0, T)$ and $\beta$ is small enough so that $t_i-\frac{3}{4}\beta>0$ and $t_i+\frac{3}{4}\beta<T$ for all $i=0,\ldots,M$. Let us choose now $t_i=t_{i-1}-\beta$ for $i=1,\ldots,M$. Note that in this case
$$(t_i+\frac{1}{4}\beta, t_i+\frac{3}{4}\beta)=(t_{i-1}-\frac{3}{4}\beta,
t_{i-1}-\frac{1}{4}\beta),$$
and in particular, $R_i^+\cap R_{i-1}^-\neq \emptyset$. Thus,
\begin{eqnarray*}
\iint_{R_{         {M}}^-}u_n(x,t)\,dx\,dt&\leq&\mbox{ ess}\inf\limits_{R_{         {M}}^+}u_{n}(x,t)\leq \mbox{ ess}\inf\limits_{R_{         {M}}^+\cap R_{{         {M}}-1}^-}u_{n}(x,t)\\
&\leq&\frac{1}{|R_{         {M}}^+\cap R_{{         {M}}-1}^-|}\iint_{R_{         {M}}^+\cap R_{{         {M}}-1}^-}u_n(x,t)\,dx\,dt\\
&\leq& \frac{1}{|R_{         {M}}^+\cap R_{{         {M}}-1}^-|}\iint_{R_{{         {M}}-1}^-}u_n(x,t)\,dx\,dt\\
&\leq&\ldots\leq C\iint_{R_0^-}u_n(x,t)\,dx\,dt\leq \tilde{C}.
\end{eqnarray*}

     Therefore, supposing $x_0=0$, by the Monotone Convergence Theorem there
exists $u\ge 0$  such that $u_n\uparrow u$ strongly in
$L^1(B_{r}(0)\times (t_1, t_2))$. Let now $\varphi$ be the solution to the problem
\begin{equation}\left\{
\begin{array}{rcl}\label{traslation}
-\varphi_t+(-\Delta)^{s}\varphi&=&\chi_{B_{r}(0)\times [t_1,t_2]}\inn\Omega\times (0,T),\\
\varphi(x,t)&>&0\inn \Omega\times (0,T),\\
\varphi(x,t)&=&0\inn (\ren\setminus\Omega)\times[ 0,T),\\
\varphi(x,T)&=&0\inn\O.
\end{array}\right.
\end{equation}
Note that, due to the regularity of the right hand sides of problems \eqref{eq:vn} and \eqref{traslation}, both $u_n$ and $\varphi$ are in the energy space, and thus both can be used as test functions in the energy formulation of the problems. Indeed, considering first $u_n$ as test function in \eqref{traslation} and then, after integrating by parts, $\varphi$ in \eqref{eq:vn}, and defining $\eta:=\inf_{B_{r}(0)\times (t_1, t_2)}{\varphi(x,t)}$, we have
\begin{eqnarray*}
C&\geq&\dint_{t_1}^{t_2}\dint_{B_{r}(0)}u_n(x,t)\,dx\,dt\\
&\geq&\l\dint_{0}^{T}\dint_{\O}\dfrac{u_{n-1}}{|x|^{2s}+\frac{1}{n}}\varphi\,dx\,dt+\dint_{0}^{T}\dint_{\O}u_{n-1}^p
\varphi\,dx\,dt+\dint_{0}^{T}\dint_{\O}T_n(f)\varphi\,dx\,dt\\
&\geq&\l\eta\dint_{t_1}^{t_2}\dint_{B_r(0)}\dfrac{u_{n-1}}{|x|^{2s}+\frac{1}{n}}\,dx\,dt+\eta\dint_{t_1}^{t_2}\dint_{B_r(0)}u_{n-1}^p\,dx\,dt+\eta\dint_{t_1}^{t_2}\dint_{B_r(0)}T_n(f)\,dx\,dt.
\end{eqnarray*}
By the Monotone Convergence Theorem,
$$
\begin{array}{rcl}
u_{n-1}^p&\rightarrow& u^p\inn L^1(B_{r}(0)\times (t_1, t_2)),\\
\dfrac{u_{n-1}}{|x|^{2s}+\frac{1}{n}}&\nearrow&\dfrac{u}{|x|^{2s}}\inn
L^{1}(B_{r}(0)\times (t_1, t_2)),\\
T_n(f)&\rightarrow& f\inn L^1(B_{r}(0)\times (t_1, t_2)).
\end{array}
$$
Thus it follows that $u$ is a weak supersolution to
\eqref{eq:prob} in {$B_{r}(0)\times (t_1, t_2)$}, a
contradiction with Theorem \ref{nonexistence}. \end{proof}

\subsection{Existence results for $1<p<p_+(\lambda,s)$.}\label{52}

The goal  now  is to
consider the complementary interval of powers, $1<p<p_+(\lambda,s)$, and to prove that under some
suitable hypotheses on $f$ and $u_0$, problem \eqref{semilinearprob} has a positive solution.
We will consider here the case
$f\equiv 0$. For the case $f\not\equiv 0$, see Remark \ref{rk:problems}.

First of all, notice that if $0<\lambda\le \Lambda_{N,s}$ and
$1<p<p_+(\lambda,s)$, the stationary problem
\begin{equation}\label{eq:equil}
 (-\Delta)^{s} u=\l\dfrac{\,u}{|x|^{2s}}+u^p \hbox{  in } \Omega,\quad u>0\mbox{ in }\Omega, \quad u=0 \hbox{  in } \mathbb{R}^{N}\setminus \Omega,
 \end{equation}
has a positive supersolution $w$, depending on the following cases:
\begin{itemize}
\item[(A)] $0<\lambda<\Lambda_{N,s}$:
In Proposition 2.3 of \cite{BMP}, the authors find a positive solution to the problem
\begin{equation}\label{eq:concconv}
 (-\Delta)^{s} u=\l\dfrac{\,u}{|x|^{2s}}+u^p +\mu u^q \hbox{  in } \Omega, \quad u>0\mbox{ in } \Omega, \quad u=0 \hbox{  in } \mathbb{R}^{N}\setminus \Omega,
 \end{equation}
 for $\mu$ small enough, $0<q<1$ and $1<p<p_+(\lambda,s)$. In particular, for every $\mu\geq 0$ this solution is a supersolution of \eqref{eq:equil}. Note that for $1<p\leq 2^{*}_{s}-1$ this supersolution is in the energy space, and for $2^{*}_{s}-1< p<p_+(\lambda,s)$, it is a weak positive supersolution.
\item[(B)] If $\l=\Lambda_{N,s}$, then
$p_+(\lambda,s)=2^{*}_{s}-1$. Thus, instead of $H_0^s(\Omega)$, we
consider the Hilbert space $H(\Omega)$ defined in
\eqref{Hardynorm}. Since $H(\Omega)$ is compactly embedded in
$L^p(\Omega)$ for all $1\leq p < 2^*_s$, classical variational
methods in the space $H(\O)$ allow us to prove the existence of
a positive solution $w$ to the stationary problem
\eqref{eq:equil}.
\end{itemize}
\begin{Theorem}\label{th:100}
Assume that $0<
\l\leq \Lambda_{N,s}$ and $1<p<p_+(\lambda,s)$. Suppose that
$u_0(x)\le \overline{w}$, where $\overline{w}$ is a supersolution to the stationary problem
$$(-\Delta)^s w=\l\dfrac{\,w}{|x|^{2s}}+w^p \mbox{ in } \O,\quad w(x)>0\mbox{ in }\Omega,\quad w(x)=0\onn\mathbb{R}^{N}\setminus \O.$$
 Then for all $T>0$, the  problem
\begin{equation}\label{eq:probexis}
\left\{
\begin{array}{rcl}
 u_t+(-\Delta)^{s} u&=&\l\dfrac{\,u}{|x|^{2s}}+u^p \mbox{ in }\O\times
 (0,T),\\
 u(x,t)&=&0\inn (\ren\setminus\Omega)\times[ 0,T),\\
u(x,0)&=&u_0(x) \mbox{ if }x\in\O,
\end{array}
\right.
\end{equation}
has a global positive solution. If $\overline{w}$ is a weak supersolution, the solution will be also weak, and likewise, if $\overline{w}$ is an energy supersolution, problem \eqref{eq:probexis} will have an energy solution.

\end{Theorem}
\begin{proof}
Since $\overline{w}(x)\ge u_0(x)$ for all $x\in \Omega$, then
$\overline{w}$ is a positive supersolution to problem
\eqref{eq:probexis}. Hence, we conclude just by applying Lemma
\ref{Comparison}, whether $\overline{w}$ is a weak supersolution,
or Lemma \ref{ApproxEnergy}, if $\overline{w}$ is an energy
supersolution.
 \end{proof}
\begin{remark}\label{rk:problems}

\

\textsc{(I)}
 With the results above we find the optimality of the power $p_+(\lambda,s)$, what was our main aim.
 Nevertheless, it could be interesting to know the optimal class of data for which there
 exists a solution and  the regularity of such solutions according to the regularity of the data.
 In this direction, considering $g=u^p+cf$ in problem \eqref{problemBG}, Theorem \ref{necesario} establishes that, necessarily,
 $$\int_{B_r(0)}|x|^{-\frac{N-2s}{2}+\alpha}u_0\,dx<+\iy,$$
 if we expect to find a solution of problem \eqref{semilinearprob}.

\textsc{(II)}
In the presence of  a source term $f\gneqq 0$,  if $f(x,t)\le \frac{c_0(t)}{|x|^{2s}}$
with $c_0(t)$ bounded and {\it sufficiently small}, then the computation above allows us to prove the
existence of a supersolution. Then the existence of a minimal solution to problem \eqref{eq:prob} follows for
all $p<p_+(\lambda,s)$.
\end{remark}

\appendix
\section{Regularity in bounded domains: Relation between weak solutions and viscosity solutions} \label{appreg}
The regularity of the solutions to the problem
\begin{equation}\label{regularity}
\left\{
\begin{array}{rcl}
 u_t+(-\Delta)^{s} u&=&f(x,t)\mbox{ in }\mathbb{R}^N\times (0,T),\\
u(x,0)&=&u_0(x) \mbox{ if }x\in \mathbb{R}^N,
\end{array}
\right.
\end{equation}
according to the regularity of the data can be seen in \cite{CF}. In this situation the properties of the fundamental solution play a crucial role.

We will try to prove, in particular, that the  class of test functions,
\begin{equation*}
\mathcal{T}:=\Big\{\phi:\mathbb{R}^N\times [0,T]\rightarrow\mathbb{R},\,\hbox{ s.t.}
\left\{\begin{array}{lll}\quad &\phi_t+(-\Delta)^s\phi=\varphi,\, \varphi\in L^\infty(\Omega\times (0,T)),\\
\quad&\phi= 0\inn (\ren\setminus \Omega)\times (0,T],\quad \phi(x,0)=0 \inn \Omega
\end{array}
\right.
 \Big\},
\end{equation*}
is a class of \textit{regular functions} for which the equation is also verified in an almost everywhere way.
In this sense we will say that the functions in $\mathcal{T}$ are classical          { or strong} solutions to the fractional heat equation.

To start we  will consider such a function $\phi$ as an energy solution to the fractional heat equation.
According to the results in \cite{LPPS} we find that every $\phi\in \mathcal{T}$ belongs in particular to
$L^\infty(\Omega\times (0,T))$ and by using the results in \cite{FK} we find that $\phi$ is Hölder continuous in  $\Omega\times (0,T)$.

Moreover in \cite{FRos} the authors prove the  following sharper result.
\begin{Theorem}\label{FRos} Assume that $\O$ is a $\mathcal{C}^{1,1}$ bounded domain in $\mathbb{R}^N$ and let $v_0\in L^2(\Omega)$.
Consider $v$ the unique weak solution to problem
\begin{equation}\label{equ:hom}
\left\{\begin{array}{rcll}
v_t+(-\Delta)^s v&=&0 \quad&\hbox{  in  } \Omega\times(0,\infty)\\
v&=&0 \quad&\hbox{  in  }\mathbb{R}^N\setminus \Omega\times(0,\infty)\\
v(x,0)&=&v_0(x)\quad&\hbox{  in  } \Omega.
\end{array}
\right.
\end{equation}
Then, for each $\epsilon>0$,
\begin{itemize}
\item[$i)$]
\begin{equation}
\sup\limits_{t>\epsilon}\|u(\cdot, t)\|_{\mathcal{C}^s}\le C_1(\epsilon) \|v_0\|_{L^2(\Omega)}.
\end{equation}
\item [$ii)$]
\begin{equation}
\sup\limits_{t>\epsilon}\|\frac{u(\cdot, t)}{\delta^s}\|_{\mathcal{C}^{s-\eta}}\le C_2(\epsilon) \|v_0\|_{L^2(\Omega)}\;\;\mbox{         {for any $\eta>0$}}.
\end{equation}
\end{itemize}
Both constants $C_1$ and $C_2$ blow up when $\epsilon\to 0$.

 In addition, for every $j\in \mathbb{N}$,
\begin{equation}
\sup\limits_{t>\epsilon}\|\dfrac{\p^ju(\cdot, t)}{\p
t^j}\|_{\mathcal{C}^s}\le C_j(\epsilon) ||v_0||_{L^2(\Omega)}.
\end{equation}

\end{Theorem}

The proof in \cite{FRos} is a consequence of the regularity up to the boundary of the normalized eigenfunction of the elliptic
Dirichlet problem and the use of separation of variables.

As a consequence of the previous result we have the following extension property.
\begin{Corollary}
 Let $v$ be as in Theorem \ref{FRos}, then $v$ can be extended as  a continuous function to $\mathbb{R}^N\times(0,\infty)$.
\end{Corollary}
We refer to  \cite{ChD2} (Definition 2.8) for the concept of viscosity solution (and the definition of the class $S$ of test functions),  and to Theorem 5.3 of the same article  to prove
the stability of viscosity solutions by uniform limits.
Therefore, since we are working in the linear setting, via a convenient mollification  we can assume that $v$ is uniform limit
on compact sets of viscosity solutions for approximated problems and thus it is a viscosity solution.

Hence, for a fixed $(x_0,t_0)\in \O\times (0,T)$, we can assume that $v$ is regular in a neighborhood of $(x_0,t_0)$.

We prove that $v$ is a viscosity subsolution (and in a similar way we prove that it is a supersolution).
Consider a test function $\phi\in S$ (see \cite{ChD2} to see the definition of $S$) such that
\begin{itemize}
\item[$i)$] $v(x_0,t_0)=\phi(x_0,t_0)$,
\item[$ii)$] $v(y,s)<\phi(y,s)$.
\end{itemize}
Then
\begin{equation*}\begin{split}
\phi_{t^-}(         {x_0},t_0)&+         {(-\Delta )^s}\phi(x_0,t_0)=\displaystyle\lim_{h\to 0^+}\frac{\phi(x_0,t_0)-\phi(x_0,t_0-h)}{h}+\int_{         {\R^N}} \frac{\phi(x_0, t_0)-\phi(y,t_0)}{|x-y|^{N+2s}}         {\,dy} \\
&\displaystyle\le\lim_{h\to 0^+}\frac{v(x_0,t_0)-v(x_0,t_0-h)}{h}+\int_{         {\R^N}} \frac{v(x_0, t_0)-v(y,t_0)}{|x-y|^{N+2s}}         {\,dy}=0
\end{split}\end{equation*}
Hence, summarizing, we obtain that $v$ is a viscosity subsolution and in a similar way supersolution.

As an application of the previous result
for $f\in L^\infty(\Omega\times(0,T))\cap\mathcal{C}^{\alpha,\beta}(\Omega\times(0,T))$, we get a similar result for the problem
\begin{equation}\label{equ:ques}
\left\{\begin{array}{rcll}
u_t+(-\Delta)^s u&=&f(x,t) \quad&\hbox{  in  } \Omega\times(0,\infty)\\
u(x,t)&=&0 \quad&\hbox{  in  } (\mathbb{R}^N\setminus\Omega)\times(0,\infty)\\
u(x,0)&=&0 \quad&\hbox{  in  } \Omega.
\end{array}
\right.
\end{equation}
Indeed,   for fixed $\tau>0$ we solve
\begin{equation}\label{equ:ques1}
\left\{\begin{array}{rcll}
v_{t}(x,t,\tau)+(-\Delta)^s v(x,t,\tau)&=&0 \quad&\hbox{  in  } \Omega\times(\tau,\infty)\\
v(x,t,\tau)&=&0 \quad&\hbox{  in  }(\mathbb{R}^N\setminus\Omega)\times(\tau,\infty)\\
v(x,\tau,\tau)&=&f(x,\tau) \quad&\hbox{  in  } \Omega.
\end{array}
\right.
\end{equation}
By separation of variables as in \cite{FRos}, we obtain
$$v(x,t,\tau)=\sum_0^\infty c_k(\tau) e_k(x)e^{-\lambda_k (t-\tau)}$$
where $(\lambda_k, e_k)$ are the eigenvalues and the normalized eigenfunction respectively of the \textit{fractional laplacian} and the Fourier coefficients are defined by
$$c_k(\tau)=\int_\Omega f(y,\tau)e_k(y)dy.$$
Then as above, $v(x,t,\tau)$ can be extended by continuity to the
whole space $\mathbb{R}^N\times (\tau,\infty)$.

It is easy to check that
$$u(x,t):=\int_0^t v(x,t,\tau) d\tau,$$
is the extension to $\mathbb{R}^N\times (0,\infty)$ of the unique solution to \eqref{equ:ques}. By mollification and by the stability of
viscosity solutions by uniform limits in compact sets, we can assume $u$ regular in the point where we want to test.
As before, assume  $\phi\in S$ verifying
\begin{itemize}
\item[$i)$] $u(x_0,t_0)=\phi(x_0,t_0)$,
\item[$ii)$] $u(y,s)<\phi(y,s)$.
\end{itemize}
Then
\begin{equation*}\begin{split}
\phi_{t^-}(         {x_0},t_0)&         {+(-\Delta )^s}\phi(x_0,t_0)=\displaystyle\lim_{h\to 0^+}\frac{\phi(x_0,t_0)-\phi(x_0,t_0-h)}{h}+\int_{         {\R^N}} \frac{\phi(x_0, t_0)-\phi(y,t_0)}{|x-y|^{N+2s}}         {\,dy}\\
&\le \displaystyle\lim_{h\to 0^+}\frac{u(x_0,t_0)-u(x_0,t_0-h)}{h}+\int_{         {\R^N}} \frac{u(x_0, t_0)-u(y,t_0)}{|x-y|^{N+2s}}         {\,dy}=f(x,t)
\end{split}\end{equation*}
Therefore,  $u$ is a viscosity subsolution. Analogously, we prove that $u$ is a viscosity supersolution.

Then we can use the regularity results in \cite{ChD1, ChD2, JX}. In particular, by using  Corollaries 2.6. 2.7 in \cite{JX},
for a second member Hölder continuous in space-time,  we find that the equation is verified in the pointwise classical meaning, i.e.,
is a strong solution in the sense of Definition 1.3 in \cite{BPSV}.

\section{Fundamental inequalities}
We explain in this Appendix the functional results used in the previous sections. Recall first that we defined in \eqref{defMuNu} the measures $\mu$,
$\nu$ as
$$d\mu:=\dfrac{dx}{|x|^{2\g}},\qquad \hbox{ and }\qquad d\nu:=
\dfrac{dxdy}{|x|^{\gamma}|y|^{\gamma}|x-y|^{N+2s}}. $$

Let begin by the next extension lemma whose proof follows using
the same arguments of \cite{Adams} (see also \cite{DPV}).
\begin{Lemma}\label{ext}
Let $\Omega\subset \ren$ be a smooth domain. Then for all
$w\in Y^{s,\g}(\Omega)$, there exists $\tilde{w}\in
Y^{s,\g}(\ren)$ such that $\tilde{w}_{|\Omega}=w$ and
$$
\|\tilde{w}\|_{Y^{s,\g}(\ren)}\le C \|{w}\|_{Y^{s,\g}(\Omega)},
$$
where $C:= C(N,s,\O, \gamma)>0$.
\end{Lemma}

     Recall that $Y^{s,\g}_0(\O)$ was defined as the completion of
$\mathcal{C}^\infty_0(\O)$ with respect to the norm of
$Y^{s,\g}(\O)$. It is clear that if $\phi\equiv C\in
Y^{s,\g}_0(\O)$, then $C\equiv 0$.

If  $\Omega$ is a bounded regular domain, we can prove the next Poincar\'e inequality.
\begin{Theorem}\label{poincare}
There exists a positive constant $C:=C(\O,N,s  {,\gamma})$ such
that for all $\phi\in \mathcal{C}^\infty_0(\O)$, we have
$$ C\dint_{\O}\phi^2(x)d\mu\le \int_{\O}
\int_{\O}(\phi(x)-\phi(y))^2d\nu.$$
\end{Theorem}
\begin{proof}
{If $\phi\equiv 0$, the inequality follows trivially.} Thus, let
us define
$$
\lambda_1(\Omega):=\inf_{\{\phi\in
\mathcal{C}^\infty_0(\O),{\phi\not\equiv
0}\}}\dfrac{\dyle\int_{\O}
\int_{\O}(\phi(x)-\phi(y))^2d\nu}{\dint_{\O}\phi^2(x)d\mu}.
$$
Hence, to prove the lemma we need to check that $\lambda_1(\Omega)>0$. We argue by
contradiction, that is, let us suppose $\lambda_1(\Omega)=0$. Then we get the
existence of $\{\phi_n\}_{n\in \mathbb{N}}\subset
\mathcal{C}^\infty_0(\O)$ such that
$$
\dint_{\O}\phi^2_n(x)d\mu=1\mbox{  and  }\int_{\O}
\int_{\O}(\phi_n(x)-\phi_n(y))^2d\nu\to 0\mbox{  as   }n\to
\infty.$$
It is clear that $\|\phi_n\|_{Y^{s,\g}(\O)}\le C$, and hence
we reach the existence of $\bar{\phi}\in Y^{s,\g}(\O)$ such that
$\phi_n\rightharpoonup \bar{\phi}$ weakly in $Y^{s,\g}(\O)$.

         {From the Sobolev inequality in \cite{AR} it follows
$$\dyle
\int_{\Omega}\dfrac{|\phi_n|^{2^*_s}}{|x|^{2^*_s\gamma}}dx\le
C(N,s,\O,\gamma)\|\phi_n\|_{Y^{s,\gamma}(\Omega)}\leq \bar{C},$$
with $\bar{C}$ independen of $n$.}
%     Define $\tilde{\phi}_n$ as the extension of $\phi_n$ given in
%Lemma \ref{ext}. Then
%$$
%\|\tilde{\phi}_n\|_{Y^{s,\g}(\ren)}\le C(N,s,\Omega,\gamma)
%\|{\phi}_n\|_{Y^{s,\g}(\Omega)}\le \bar{C}.
%$$
%
%
%From \cite{AR}  and applying Sobolev inequality, we obtain that
%$$
%\dyle
%\Big(\int_{\ren}\dfrac{|\tilde{\phi}_n|^{2^*_s}}{|x|^{2^*_s\gamma}}dx\Big)^{\frac{1}{2^*_s}}\le
%C \|\tilde{\phi}_n\|_{Y^{s,\g}(\ren)}\le \bar{C}_1.
%$$
%Therefore, we reach that $\dyle
%\int_{\Omega}\dfrac{|\phi_n|^{2^*_s}}{|x|^{2^*_s\gamma}}dx\le
%C(N,s,\O,\gamma).$

     Using the fact that $Y^{s,\g}(\O)\subset Y^{s,0}(\O)$, it follows from
\cite{Adams} (see also \cite{DPV}) that $\phi_n\to
\bar{\phi}$ strongly in $L^2(\O)$. Hence, combining the estimates above and using Vitali's Lemma we
obtain that, up to a subsequence,
$$
\phi_n\to \bar{\phi}\mbox{  strongly  in   }L^2(\Omega,d\mu),
$$
and thus,
\begin{equation}\label{int1}
\dint_{\O}{\bar{\phi}}^2(x)d\mu=1.
\end{equation}
Since $\|\bar{\phi}\|_{Y^{s,\g}(\Omega)}\le
\|\phi_n\|_{Y^{s,\g}(\Omega)}$, taking into consideration that
$\int_{\O} \int_{\O}(\phi_n(x)-\phi_n(y))^2d\nu\to 0$ as $n\to
\infty$, we get
$$
\phi_n\to \bar{\phi}\mbox{  strongly  in   } Y^{s,\g}(\O),\mbox{
thus }\int_{\O} \int_{\O}(\bar{\phi}(x)-\bar{\phi}(y))^2d\nu=0.
$$
Hence $\bar{\phi}\equiv C$. Now,   since $\bar{\phi}\in
Y^{s,\g}_0(\Omega)$, necessarily $\bar{\phi}\equiv 0$, a
contradiction with \eqref{int1}.
\end{proof}

As a direct application of Theorem \ref{poincare} we obtain that
if $\Omega$ is a bounded regular domain, then every $w\in
Y^{s,\g}_0(\Omega)$ satisfies
$$
\|\tilde{w}\|_{Y^{s,\g}(\ren)}\le C
\Big(\dint_{\O}\dint_{\O}(w(x)-w(y)|^2d\nu\Big)^{\frac 12},
$$
where $C:=C(N,s,\O,\gamma)>0$ and $\tilde{w}$ is the extension of $w$
given in Lemma \ref{ext}.

     Define now the operator $$ L_{\g,\O} (w)(x):=
  {a_{N,s}}\, P.V. \int_{\O} (w(x)-w(y))K(x,y)dy,\hbox{ where }{K(x,y):=\frac{1}{|x|^{\gamma}|y|^{\gamma}|x-y|^{N+2s}}.}
$$

     In the case $\Omega=\ren$, we have the next result.
\begin{Lemma}\label{hardy001}
If $w(x):=|x|^{-\theta}$, with
$0<\theta<(N-2s-2\g)$, then there exists a positive
constant          {$C:=C(N,s,\g,{\theta})$} such that
$$
L_{\g,\ren} (w)(x)= C\frac{w(x)}{|x|^{2s+2\g}}\:\:\: a.e\: \mbox{
in }         {\R^N}.
$$
\end{Lemma}
\begin{proof}
In $\ren$, the operator has the form
$$
L_{\g,\ren} (w)(x):= {a_{N,s}} P.V. \int_{\ren}
\dfrac{(w(x)-w(y))}{|x|^{\gamma}|y|^{\gamma}|x-y|^{N+2s}}dy.
$$
As in the proof of Theorem 5.5, we  closely follow the arguments used in \cite{FV}.

By setting $r:=|x|$ and $\rho:=|y|$, then $x=rx'$, and $y=\rho y'$ where
$|x'|=|y'|=1$. Thus,
$$
\begin{array}{rcl}
L_{\g,\ren}(w)(x)&=& \dfrac{a_{N,s}}{|x|^{\g}}
\dint\limits_0^{+\infty}\dfrac{(r^{-{\theta}}-\rho^{-{\theta}})\rho^{N-1}}{\rho^{\gamma}
r^{N+2s}}\left(
\dint\limits_{|y'|=1}\dfrac{dH^{         {N}-1}(y')}{|x'-\dfrac{\rho}{r}
y'|^{N+2s}} \right) \,d\rho.
\end{array}
$$
Set now $\sigma:=\dfrac{\rho}{r}$. Then
$$
L_{\gamma,
\ren}(w)(x)=\dfrac{a_{N,s}\, w(x)}{|x|^{2s+2\gamma}}\dint\limits_0^{+\infty}
(1-\sigma^{-{\theta}})\sigma^{N-\g-1}
\left(\dint\limits_{|y'|=1}\dfrac{dH^{         {N}-1}(y')}{|x'-\sigma
y'|^{N+2s}} \right) \,d\sigma.
$$
Define
$$ K(\sigma):=\dint\limits_{|y'|=1}\dfrac{dH^{         {N}-1}(y')}{|x'-\sigma
y'|^{N+2s}},
$$
then
\begin{equation}\label{kkk}
K(\sigma)=2\frac{\pi^{\frac{N-1}{2}}}{\Gamma(\frac{N-1}{2})}\int_0^\pi
\frac{\sin^{N-2}(\eta)}{(1-2\sigma \cos
(\eta)+\sigma^2)^{\frac{N+2s}{2}}}d\eta.
\end{equation}
Thus
$$
L_{\gamma, \ren}(w)
=\Lambda_{N,s,\gamma}\dfrac{w(x)}{|x|^{2s+2\gamma}},
$$
where
$$
\Lambda_{N,s,\gamma}=a_{N,s}\dint_0^{+\infty}
(\sigma^{\theta}-1)\sigma^{N-\g-\theta-1}
K(\sigma)d\sigma.
$$
As in \cite{FV}, taking into consideration the behavior of $K$
near $\sigma={1}$ and at $+\infty$, we can prove that
$|\Lambda_{N,s,\gamma}|<\infty$. To conclude we just have to show
that $\Lambda_{N,s,\gamma}>0$.

     Since $K(\frac{1}{s})=s^{N+2s}K(s)$ for all $s>0$, we get
\begin{equation*}
\begin{array}{lll}
\Lambda_{N,s,\gamma} &= &\dyle \int_0^{1}
(\sigma^{\theta}-1)\sigma^{N-\g-\theta-1} K(\sigma)d\sigma +
\int_1^{\infty} (\sigma^{\theta}-1)\sigma^{N-\g-\theta-1}
K(\sigma)d\sigma\\
&&\\
&= &-\dyle \int_{1}^\infty (\xi^{\theta}-1)\xi^{2s+\g-1}
K(\xi)d\xi+\int_1^{\infty}
(\sigma^{\theta}-1)\sigma^{N-\g-\theta-1}
K(\sigma)d\sigma\\
&&\\
&=& \dyle\int_1^{\infty}K(\sigma)
(\sigma^{\theta}-1)(\sigma^{N-\gamma-\theta-1}-\sigma^{2s+\g-1})d\sigma.
\end{array}
\end{equation*}
Since $0<\theta<N-2s-2\gamma$, then the results follows.

\end{proof}

Next we formulate an extension of a well-known Picone identity, that in the case of regular functions and the Laplacian operator
was obtained by Picone in \cite{Pi} (see  \cite {AP}  for an integral extension related to positive Radon measures).

\begin{Theorem}{\it (Picone's Type Inequality).}\label{Picone}
Consider $u, v\in H_0^s(\Omega)$, where  $(-\Delta)^s u=\tilde{\nu}$ is a bounded Radon measure in $\Omega$, and $u\gneq 0$. Then,
\begin{equation}\label{81}
\dint_{\O} \dfrac{(-\Delta)^s u}{u} v^2\,dx\leq \frac{a_{N,s}}{2}\|v\|_{H_0^s(\O)}^2.
\end{equation}
\end{Theorem}
     See \cite{LPPS} for a proof. It is worthy to point out  that the proof relies in a pointwise inequality.
Therefore, we can reformulate the Picone inequality as follows.
\begin{Corollary}\label{pic}
Let $w\in          {Y^{s,\g}(\O)}$ be such that $w>0$ in $\O$. Assume that
$L_{\g, \O}(w)= \tilde{\nu}$ with $\tilde{\nu}\in L^1_{loc}(\ren)$ and $\tilde{\nu}\gneqq
0$, then for all $u\in \mathcal{C}^\infty_0(\O)$, we have
\begin{equation}\label{picone1}
\frac {  {a_{N,s}}}{2}
\dint_{\O}\dint_{\O}\dfrac{|u(x)-u(y)|^{2}}{|x-y|^{N+2s}}\dfrac{dxdy}{|x|^\g|y|^\g}\ge
\langle L_{\g,\O} (w),\frac{u^2}{w}\rangle_{Y_{0}^{s,\gamma}(\Omega)}.
\end{equation}
\end{Corollary}

As a consequence we get the next Hardy type inequality.
\begin{Theorem}\label{generalh}
There exists a positive constant $C(N,s  {,\gamma})$
such that for all $\phi\in \mathcal{C}^\infty_0(\ren)$, we have
$$ C\dint_{\ren}\frac{\phi^2(x)}{|x|^{2s+2\g}}dx\le \int_{\ren}
\int_{\ren}(\phi(x)-\phi(y))^2d\nu.$$
\end{Theorem}
\begin{proof}
Let $\phi\in \mathcal{C}^\infty_0(\ren)$ and define
$w(x):=|x|^{-\theta}$, with $0<\theta<\frac{N-2s-2\g}{2}$. Then, by Lemma \ref{hardy001},
$$
L_{\g,\ren} (w)(x)= C\frac{w(x)}{|x|^{2s+2\g}}\:\:\: a.e\:
\mbox{ in }\ren.
$$
Using \eqref{picone1} it follows that
$$
\dint_{\ren}\dint_{\ren}\dfrac{|\phi(x)-\phi(y)|^{2}}{|x-y|^{N+2s}}\dfrac{dxdy}{|x|^\g|y|^\g}\ge
\langle L_{\g,\ren} (w),\frac{\phi^2}{w}\rangle_{Y_{0}^{s,\gamma}(\Omega)}=
C(N,s,\gamma)\dint_{\ren}\frac{\phi^2(x)}{|x|^{2s+2\g}}dx.$$ Hence
we conclude.
\end{proof}

In the case where $\Omega$ is a bounded domain, we have:

\begin{Theorem}\label{hardyGG}
There exists a positive constant $C(\O,N,s,\gamma)$  such that for
all $\phi\in \mathcal{C}^\infty_0(\O)$, we have
$$ C\dint_{\O}\frac{\phi^2(x)}{|x|^{2s+2\g}}dx\le \int_{\O}
\int_{\O}(\phi(x)-\phi(y))^2d\nu.$$
\end{Theorem}

\begin{proof}
Let $\phi\in \mathcal{C}^\infty_0(\Omega)$ and define
$\tilde{\phi}$ to be the extension of $\phi$ to $\ren$ given in
Lemma \ref{ext}. Then from Theorem \ref{generalh}, we get
$$
\dint_{\ren}\dint_{\ren}\dfrac{|\tilde{\phi}(x)-\tilde{\phi}(y)|^{2}}{|x-y|^{N+2s}}\dfrac{dxdy}{|x|^\g|y|^\g}\ge
C(N,s,\gamma)\dint_{\ren}\frac{\tilde{\phi}^2(x)}{|x|^{2s+2\g}}dx.$$
Now, using the fact that $\tilde{\phi}_{|\Omega}=\phi$ and combining
the results of Lemma \ref{ext} and Theorem \ref{poincare}, we
reach the desired result.
\end{proof}

In the case of nonzero boundary conditions, we obtain the following version of the Hardy inequality.
\begin{Theorem}\label{general000}
There exists a positive constant $C(\O,N,s,\gamma)$ such that for all
$\phi\in Y^{s,\g}(\O)$, we have
$$ C\dint_{\Omega}\frac{\phi^2(x)}{|x|^{2s+2\g}}dx\le \int_{\O}
\int_{\Omega}(\phi(x)-\phi(y))^2d\nu+ \dint_{\Omega}\phi^2(x)d\mu.
$$
\end{Theorem}
\begin{proof}
Fix $\phi\in Y^{s,\g}(\O)$ and define $\tilde{\phi}$ as the
extension of $\phi$ to $\ren$ given in Lemma \ref{ext}. Then from
Theorem \ref{generalh}, we get
$$
\dint_{\ren}\dint_{\ren}\dfrac{|\tilde{\phi}(x)-\tilde{\phi}(y)|^{2}}{|x-y|^{N+2s}}\dfrac{dxdy}{|x|^\g|y|^\g}\ge
C(N,s,\gamma)\dint_{\ren}\frac{\tilde{\phi}^2(x)}{|x|^{2s+2\g}}dx\ge
C(N,s,\gamma)\dint_{\Omega}\frac{\phi^2(x)}{|x|^{2s+2\g}}dx.
$$ Since $\|\tilde{\phi}\|_{Y^{s,\g}(\ren)}\le C(\O)
\|\phi\|_{Y^{s,\g}(\Omega)}$, then the result follows.
\end{proof}

Notice now that when $\phi\in \mathcal{C}^\infty_0(\O)$, the
following Sobolev inequality holds (see \cite{AR}),
\begin{equation}\label{Sob}
\Big(\dint\limits_{\O}
\dfrac{|\phi(x)|^{2^*_{s}}}{|x|^{\g
2^*_s}}\,dx\Big)^{\frac{2}{2^*_{s}}}\le
C(\O,N,s,\g)\dint_{\O}\dint_{\O}(\phi(x)-\phi(y))^2d\nu.
\end{equation}
Moreover, in the particular case $\phi\in Y^{s,\g}(B_R)$, as an application of Theorem \ref{general000} we can prove the following improved inequality.
\begin{Theorem}\label{sobolev}
Let {$R>0$} and $\phi\in Y^{s,\g}(B_R)$. Then, there exists
$C:=C(N,s,R,\gamma)>0$ such that
\begin{equation}\label{sobol}
C\dyle \Big(\int_{B_R}
\dfrac{|\phi|^{2^*_s}}{|x|^{2^*_s\g}}dx\Big)^{\frac{2}{2^*_s}}\le
\int_{B_R} \int_{B_R}(\phi(x)-\phi(y))^2d\nu+ R^{-2s}
\int_{B_R}\phi^2d\mu.
\end{equation}
\end{Theorem}
\begin{pf}
We prove the result for $R=1$, and then \eqref{sobol}
follows by a scaling argument. We set
$\phi_1(x):=\frac{\phi(x)}{|x|^\g}$. Then from
\cite{Adams} we know that
\begin{equation}\label{sobol1}
C\dyle \Big(\int_{B_1}
|\phi_1|^{2^*_s}dx\Big)^{\frac{2}{2^*_s}}\le \int_{B_1}
\int_{B_1}\dfrac{(\phi_1(x)-\phi_1(y))^2}{|x-y|^{N+2s}}dxdy+
\int_{B_1}\phi_1^2dx.
\end{equation}
To get the desired result we just have to estimate the term
$$
\int_{B_1}
\int_{B_1}\dfrac{(\phi_1(x)-\phi_1(y))^2}{|x-y|^{N+2s}}dx         {dy}.
$$
Since
\begin{eqnarray*}
(\phi_1(x)-\phi_1(y))^2 &=&
\dfrac{(\phi(x)-\phi(y))^2}{|x|^{\g}|y|^{\g}}  {+}\left(\frac{\phi^{  {2}}(x)}{|x|^{\g}}-\frac{\phi^{  {2}}(y)}{|y|^{\g}}\right)\left(\frac{1}{|x|^{\g}}-\frac{1}{|y|^{\g}}\right),
\end{eqnarray*}
it follows that
$$
\dyle\int_{B_1}
\int_{B_1}\dfrac{(\phi_1(x)-\phi_1(y))^2}{|x-y|^{N+2s}}dx  {dy}\le
\dyle\int_{B_1}
\int_{B_1}\dfrac{(\phi(x)-\phi(y))^2}{|x|^{\g}|y|^{\g}|x-y|^{N+2s}}dx
{dy}+ \int_{B_1} L_{0,{B_1}}(|x|^{-\g})\dfrac{\phi^2(x)}{|x|^{
{\g}}}dx.
$$
Proceeding as in the proof of Lemma \ref{hardy001}, since $0<\gamma<\dfrac{N-2s}{2}$, we can prove that $ L_{0,
{B_1}}(|x|^{-\g})\le \dfrac{C}{|x|^{\g+2s}}$, and hence
$$\dyle\int_{B_1}
\int_{B_1}\dfrac{(\phi_1(x)-\phi_1(y))^2}{|x-y|^{N+2s}}dx  {dy}\le
 \dyle \int_{B_1}
\int_{B_1}\dfrac{(\phi(x)-\phi(y))^2}{|x|^{\g}|y|^{\g}|x-y|^{N+2s}}dx  {dy}+
C\int_{B_1}\dfrac{\phi^2}{|x|^{2s+2\g}}{dx}.
$$
Finally, using Theorem \ref{general000} and substituting
$\phi(x)=|x|^{\gamma}\phi_1(x)$, we reach
\eqref{sobol}.
\end{pf}

We state now a weighted version of the Poincar\'e-Wirtinger
inequality used in the proof of Lemma \ref{log}.

\begin{Theorem}\label{PW}
Let $w\in Y^{s,\g}(  {B_1})$ and assume that $\psi$ is a radial
decreasing function such that $\text{supp}\:\psi\subset B_1$ and
$0\lneqq \psi\le 1$. Define
$$
W_\psi:=\dfrac{\dyle\int_{B_{1}}w(x)\psi(x)d\mu}{\dyle\int_{B_1}\psi(x)d\mu}.
$$
Then, there exists $C:=C(N,s,\psi)>0$ such that
$$ \int_{B_1}(w(x)-W_\psi)^2\psi(x)d\mu\le C\int_{B_1}
\int_{B_1}(w(x)-w(y))^2\min\{\psi(x),\psi(y)\}d\nu.
$$
\end{Theorem}
\begin{proof}
Define $\Psi(x):=\dfrac{\psi(x)}{|x|^{2\g}}$, that is
a radial decreasing function. Then using \cite[Corollary 6]{DyK} we get
$$ \int_{B_1}(w(x)-\bar{W}_\Psi)^2\Psi(x)dx\le
C\int_{B_1}
\int_{B_1}\dfrac{(w(x)-w(y))^2}{|x-y|^{N+2s}}\min\{\Psi(x),\Psi(y)\}dxdy,
$$
where
$$\bar{W}_\Psi=\dfrac{\int_{B_1}w(x)\Psi(x)dx}{\int_{B_1}\Psi(x)dx}.
$$
Substituting $\Psi$ by its value, we get
$$
\int_{B_1}(w(x)-\bar{W}_\Psi)^2\Psi(x)dx=\int_{B_1}(w(x)-W_\psi)^2\psi(x)d\mu,
$$
and
$$
\int_{B_1}
\int_{B_1}\dfrac{(w(x)-w(y))^2}{|x-y|^{N+2s}}\min\{\Psi(x),\Psi(y)\}dxdy=
\int_{B_1}
\int_{B_1}\dfrac{(w(x)-w(y))^2}{|x-y|^{N+2s}}
\min\Big\{\frac{\psi(x)}{|x|^{2\g}}, \frac{\psi(y)}{|y|^{2\g}}\Big\}dxdy.$$
Hence, to finish we just have to
show that
$$
\min\Big\{\frac{\psi(x)}{|x|^{2\g}},
\frac{\psi(y)}{|y|^{2\g}}\Big\}\le\frac{\min\{\psi(x),\psi(y)\}}{|x|^\g|y|^\g}
\mbox{  in  }B_1\times B_1.
$$
Without loss of generality we can assume that $|x|\ge |y|$.

     Define $H(s):=\frac{\psi(s)}{s^{2\g}}$, that is a decreasing
function in $(0,1)$. Let $s_1:=|x|$ and $s_2:=|y|$, then
$$
\min\Big\{\frac{\psi(x)}{|x|^{2\g}}, \frac{\psi(y)}{|y|^{2\g}}\Big\}=H(s_1).$$ Using that $\psi$ is
decreasing, we obtain that $\psi(s_1)\le \psi(s_2)$. Thus
$$
\frac{\min\{\psi(x),\psi(y)\}}{|x|^\g|y|^\g}=\frac{\psi(s_1)}{s_1^{\g}s_2^{\g}}.
$$
Since $s_2\le s_1\le 1$, we conclude that $H(s_1)\le
\dfrac{\psi(s_1)}{s_1^{\g}s_2^{\g}}$ and the result follows.
\end{proof}


\begin{thebibliography}{ABCD}
\bibitem{AR} {\sc B. Abdellaoui, R. Bentiffour}, {\em Caffarelli-Kohn-Nirenberg type inequalities of fractional order and applications}. Preprint.

\bibitem {AP} {\sc B. Abdellaoui,  I. Peral}, {\em Existence and nonexistence results for
quasilinear elliptic equations involving the {\it{ p-Laplacian }}
with a critical potential.} Ann. di Mat. Pura e Applicata, 182, 247--270, (2003).

\bibitem{APP} \textsc{B. Abdellaoui, I. Peral, A. Primo,} {\em Influence of the Hardy potential in a semilinear heat equation}, Proceedings of the Royal Society of Edinburgh, Section A Mathematics,  139,  (2009), no 5, 897-926.

\bibitem{APP2} \textsc{B. Abdellaoui, I. Peral, A. Primo,} {\em A remark on the fractional Hardy inequality with a remainder term}. C. R. Math. Acad. Sci. Paris 352 (2014), no. 4, 299--303.

\bibitem{AS} \textsc{M. Abramowitz, and I. A. Stegun}, \textit{Handbook of
mathematical functions with formulas, graphs, and mathematical tables.}
National Bureau of Standards Applied Mathematics Series \textbf{55}.

\bibitem{Adams} \textsc{R. A. Adams} \emph{Sobolev spaces}, Academic Press, New York, 1975.

\bibitem{A} {\sc E. Artin}, {\em The Gamma Function}, in Rosen, Michael (ed.) Exposition by Emil Artin: a selection; History of Mathematics 30. Providence, RI: American Mathematical Society (2006). \bibitem{AY} {\sc W. Allegretto, Y. X. Huang}, {\em A Picone's identity for the $p-Laplacian $ and applications}, Nonlinear Ana.  T.M.P.   32
 (1998), no 7 819--830.

\bibitem{BaGo} \textsc{P. Baras, J.A. Goldstein}, {\em The heat equation with a singular potential.} Trans. Amer. Math. Soc.
    {\bf284} (1984), no. 1, 121--139.

\bibitem{BMP} {\sc B. Barrios, M. Medina, I. Peral}, {\em Some remarks on the solvability of non-local elliptic problems with the Hardy potential}, Communications in Contemporary Mathematics. Available online, DOI: 10.1142/S0219199713500466 (2013).

\bibitem{BPSV}\textsc{{B. Barrios, I. Peral, F. Soria, E. Valdinoci}}, {\em A Widder's type Theorem for the heat equation with nonlocal diffusion.} Arch. Ration. Mech. Anal. 213 (2014), no. 2, 629--650.

\bibitem{B} {\sc W. Beckner}, {\em Pitt's inequality and the uncertainty principle}, Proceedings of the American Mathematical Society, Volume 123, Number 6 (1995).


\bibitem{BG} \textsc{E. Bombieri, E. Giusti},  \emph{Harnack's inequality for elliptic differential equations on minimal surfaces},  Invent. Math. 15  (1972), 24-46.

\bibitem{BC} {\sc H. Brezis,  X. Cabr\'{e}}, {\em Some simple nonlinear PDE's without solutions}. Boll. Unione Mat. Ital. 1-B (1998), 223-262.


\bibitem{CF} {\sc L. Caffarelli, A. Figalli}, {\em Regularity of solutions to the parabolic fractional obstacle problem}, Journal f\"{u}r die Reine und Angewandte Mathematik, {680} (2013), 191-233.



\bibitem{ChD1} \textsc{ H. Chang-Lara, G. Dávila}, \emph{Regularity for solutions of non local parabolic equations}. Calc. Var. Partial Differential Equations 49 (2014), no. 1-2, 139-172.

\bibitem{ChD2} \textsc{ H. Chang-Lara, G. Dávila},\emph{ Regularity for solutions of nonlocal parabolic equations II.} J. Differential Equations 256 (2014), no. 1, 130-156.

\bibitem{CFr} \textsc{F. M. Chiarenza, M. Frasca}, \emph{{ Boundedness for the solutions
of a degenerate parabolic equation}}, Applicable Anal. { 17} (1984), no. 4, 243-261.

\bibitem{Cs}  \textsc{F. M. Chiarenza, R. P. Serapioni}, \emph{{ A Harnack inequality
for degenerate parabolic equations.}} Comm. in PDE, { 9}
 (1984), no. 8, 719-749.

\bibitem{DaA}  {\sc A. Dall'Aglio}, {\em Approximated solutions of equations with $L\sp 1$ data. Application to the $H$-convergence of quasi-linear parabolic equations},  Ann. Mat. Pura Appl.,  { 170}  (1996), 207--240.

\bibitem{DPV} {\sc E. Di Nezza, G. Palatucci,  E. Valdinoci}, {\em Hitchhiker's guide to the fractional Sobolev
    spaces}, Bull. Sci. math. {\bf 136} (2012), no. 5, 521-573.

\bibitem{DyK} \textsc{B. Dyda, M. Kassmann}, \emph{On weighted Poincar\'e inequalities},  Ann. Acad. Sci. Fenn. Math. \textbf{38} (2013), no. 2, 721-726.

\bibitem{F} {\sc M. M. Fall}, {\em Semilinear elliptic equations for the fractional Laplacian with Hardy potential},
 Preprint. arXiv:1109.5530v4 [math.AP].

\bibitem{FK} {\sc M. Felsinger, M. Kassmann}, {\em Local regularity for parabolic nonlocal operators}. Comm. PDE, { \bf 38}  (2013) 1539--1573.


\bibitem{FRos} \textsc{X. Fernandez-Real,  X. Ros-Oton} \emph{Boundary regularity for
the fractional heat equation}, Arxiv Math.


\bibitem{FV} {\sc F. Ferrari, I.E. Verbitsky}, {\em Radial fractional Laplace operators and Hessian inequalities}. J. Differential Equations {\bf 253} (2012), no. 1, 244--272.

\bibitem{FLS} {\sc R. Frank, E. H. Lieb, R. Seiringer}, {\em Hardy-Lieb-Thirring inequalities for fractional Schr\"{o}dinger operators}, Journal of the American Mathematical Society (2008), Vol. 20, No. 4, 925-950.

\bibitem{Frank} {\sc R. Frank}, {\em A simple proof of Hardy-Lieb-Thirring inequalities},
Comm. Math. Phys. 290 (2009), No. 2, 789 - 800.

\bibitem{GW}\textsc{ C. E. Guti\'errez, R. L. Wheeden}, \emph{{ Mean value Harnack
inequalities for degenerate parabolic equations.}} Colloquium Mathematicum {\bf 60/61} (1990), no. 1, 157-194.


\bibitem{He} {\sc I. W. Herbst}, {\em Spectral theory of the operator $(p^2+m^2)^{1/2}-Ze^2/r$}, Commun. math. Phys. {\bf 53} (1977), 285-294.

\bibitem{IS} \textsc{C. Imbert, L. Silvestre}, \emph{Introduction to fully nonlinear parabolic equations}, preprint

\textit{ http://www.ma.utexas.edu/mediawiki/index.php/Starting\_page.}

\bibitem{JX} \textsc{T. Jin, J. Xiong}, \emph{Schauder estimates for solutions of linear parabolic integro-differential equations}, arXiv:1405.0758  October, 2014.


\bibitem{L} {\sc N. Landkof}, {\em Foundations of modern potential theory},
Die Grundlehren der mathematischen Wissenschaften, Band 180. Springer-Verlag,
New York-Heidelberg, 1972.

\bibitem{LPPS} {\sc T. Leonori, I. Peral, A. Primo, F. Soria}, {\em Basic estimates for solutions of a class of nonlocal elliptic and parabolic equations}. Submitted.

\bibitem{MPP} {\sc B. Abdellaoui, M. Medina, I. Peral, A. Primo } \emph{A note on the effect of the Hardy potential in some Calderón-Zygmund properties for the fractional Laplacian}, preprint.

\bibitem{MI} {\sc A. N. Milgram}, {\em Supplement II in Partial Differential Equations: Vol III in Lectures in Applied Mathematics,}  pp 229-229, L. Bers, F John and M. Schechter,  editors. Interscience New York,  1964.


\bibitem{Mos} {\sc J. Moser}, {\em On a pointwise estimate for parabolic differential equations,}  Comm. Pure Appl. Math. 17(1971). pp  101-134.

%\bibitem{Mu}  {\sc F. Murat}, {\em Soluciones renormalizadas de EDP elipticas no
%lineales}, Preprint 93023, Laboratoire d'Analyse Num\'erique de l'Universit\'e Paris VI %%(1993).

\bibitem{Pi} {\sc M. Picone}, {\em Sui valori eccezionali di un paramtro da cui
dipende una equazione differenziale lineare ordinaria del secondo
ordine.,} Ann. Scuola. Norm. Pisa. 11,(1910)1-144.

\bibitem{PR} {\sc  A. Prignet}, {\em Existence and uniqueness of "entropy" solutions of parabolic problems with $L^1$ data}, Nonlinear Anal. 28 (1997), no. 12, 1943--1954.


\bibitem{SC} {\sc L. Saloff-Coste}, {\em Aspects of Sobolev-type inequalities}, London Mathematical Society Lecture Note Series. Vol. 289.

\bibitem{S} {\sc L. Silvestre}, {\em Regularity of the obstacle  problem for a fractional power of the Laplace operator}, Comm. Pure Appl. Math. 60 (2007), no. 1, 67-112.

%\bibitem{SE} {\sc R. Servadei,  E. Valdinoci}, {\em Mountain Pass solutions for non-local %elliptic operators}, J. Math. Anal. Appl. {\bf 389} (2012), no. 2, 887-898.

\bibitem{SEV} {\sc R. Servadei,  E. Valdinoci}, {\em Weak and viscosity solutions of the fractional Laplace equation}, Pub. Mat. {\bf 58} (2014), 133-154.

\bibitem{SW}{\sc E. M. Stein, G. Weiss}, {\em Fractional integrals on n-dimensional Euclidean
space}. J. Math. Mech. 7 (1958),  503-514.

\bibitem{V} {\sc M. I. Vishik}, {\em Mixed boundary problems}. Dokl. Akad. Nauk SSSR, {\bf 97}, 193-6 (1954).

\bibitem{Y} {\sc D. Yafaev}, {\em Sharp constants in the Hardy-Rellich inequalities}. J. Functional Analysis 168 (1999), no. 1, 121--144.


\end{thebibliography}
\end{document}